\documentclass[11pt]{article}
\usepackage[papersize={8.5in,11in},top=0.9in,left=1in,right=1in,bottom=0.92in]{geometry}
\usepackage{cite}
\usepackage{epsfig}

\usepackage{url,amsmath,amsthm,amsopn,amstext,amsfonts}
\usepackage{hyperref}
\usepackage{graphicx}
\usepackage{wrapfig}
\usepackage{times}
\usepackage{marvosym}
\usepackage{appendix}

\usepackage{amssymb,amsmath,amsthm,amsopn,amstext,amsfonts}
\usepackage{bm}

\usepackage{comment}
\usepackage{soul}
\usepackage{color}

\usepackage{algorithm}
\usepackage{algpseudocode}

\usepackage{arydshln}

\newcommand{\minimize}[1]{\displaystyle\minim_{#1}}
\newcommand{\minim}{\mathop{\hbox{\rm minimize}}}
\newcommand{\minimi}[1]{\displaystyle\mini_{#1}}
\newcommand{\mini}{\mathop{\hbox{\rm min}}}
\newcommand{\maximize}[1]{\displaystyle\maxim_{#1}}
\newcommand{\maxim}{\mathop{\hbox{\rm maximize}}}
\newcommand{\maximi}[1]{\displaystyle\maxi_{#1}}
\newcommand{\maxi}{\mathop{\hbox{\rm max}}}
\newcommand{\sbt}{\text{subject\ to}}

\newcommand{\norm}[1]{\left\Vert#1\right\Vert}

\newcommand{\Real}{\mathbb R}

\newcommand{\amin}{\mathop{\mbox{argmin}}}

\newtheorem{theorem}{Theorem}

\newtheorem{lemma}{Lemma}
\newtheorem{proposition}{Proposition}
  
\newtheorem{definition}{Definition}

\theoremstyle{definition}
\newtheorem{example}{Example}

\allowdisplaybreaks

\title{Multi-Agent Decentralized Network Interdiction Games}

\author{Harikrishnan Sreekumaran\thanks{School of Industrial
    Engineering, Purdue University, West Lafayette, Indiana, USA. Email:
    hsreekum@purdue.edu} \and  
  Ashish R.~Hota\thanks{School of Electrical and Computer Engineering, Purdue University,
  West Lafayette, Indiana, USA. Email: ahota@purdue.edu} \and
  Andrew L.~Liu\thanks{Corresponding author. School of Industrial
    Engineering, Purdue University, West Lafayette, Indiana, USA. Email:
    andrewliu@purdue.edu} \and  
  Nelson A.~Uhan\thanks{Mathematics Department, United States Naval
    Academy, Annapolis, Maryland, USA. Email: uhan@usna.edu}
\and
  Shreyas Sundaram\thanks{School of Electrical and Computer Engineering, Purdue University,
  West Lafayette, Indiana, USA. Email: sundara2@purdue.edu}
  }

\begin{document}
\maketitle

\begin{abstract}
In this work, we introduce \emph{decentralized network interdiction games},
which model the interactions among multiple interdictors with differing
objectives operating on a common network.  As a starting point,  we focus on
\emph{decentralized shortest path interdiction} (DSPI) games, where multiple
interdictors try to  increase the shortest path lengths of their own
adversaries, who all attempt to traverse a common network. We first establish
results regarding the existence of equilibria for DSPI games under both discrete
and continuous interdiction strategies. To compute such an equilibrium, we
present a reformulation of the DSPI games, which leads to a generalized Nash
equilibrium problem (GNEP) with non-shared constraints. While such a problem is
computationally challenging in general,  we show that under continuous
interdiction actions, a DSPI game can be formulated as a linear complementarity
problem and solved by Lemke's algorithm.  In addition, we present decentralized
heuristic algorithms based on best response dynamics for games under both
continuous and discrete interdiction strategies.  Finally, we establish
theoretical bounds on the worst-case efficiency loss of equilibria in DSPI
games, with such loss caused by the lack of coordination among noncooperative
interdictors, and use the decentralized algorithms to empirically study the
average-case efficiency loss.
\end{abstract}



\section{Introduction}

In an interdiction problem, an agent attempts to limit the actions of an
adversary operating on a network by intentionally disrupting certain components
of the network. Such problems are usually modeled in the framework of
leader-follower games and can be formulated as bilevel optimization problems.
Interdiction models have been used in various military and homeland security
applications, such as dismantling drug traffic networks
\cite{wood1993deterministic}, preventing nuclear smuggling
\cite{morton2007models} and planning tactical air strikes
\cite{ghare1971optimal}.  Interdiction models have also found applications in
other areas such as controlling the spread of pandemics
\cite{assimakopoulos1987network} and defending attacks on computer
communication networks \cite{smith_liu_algo_2008}.

Traditionally, interdiction problems have been analyzed from a
centralized perspective; namely, a single agent is assumed to
analyze, compute and implement interdiction strategies. In many
situations, however, it might be desirable and even necessary to
consider an interdiction problem from a decentralized perspective. For
instance, a supervising body, in control of multiple agents in a common
system, may assign each agent to an adversary of interest. Each agent is
then responsible for computing and implementing its own interdiction
strategy against the designated adversary. Other situations may involve
multiple independent agents, such as security agencies of different
countries, trying to achieve a common goal on a shared network.  Without
any coordination between the agents, one might expect that a
decentralized interdiction strategy may be inefficient compared to one
determined by a central decision maker. This paper is focused on
modeling and analyzing such settings and the inefficiencies that may
arise.

In this paper, we introduce \emph{decentralized network interdiction
  (DNI) games}, in which multiple agents with differing objectives are
interested in interdicting parts of a common network. We focus on a
specific class of these games, which we call \emph{decentralized shortest path
  interdiction (DSPI) games}. We investigate various properties of
equilibria in DSPI games, including their existence and uniqueness, and
propose algorithms to compute equilibria of these games.
Using these algorithms, we also conduct empirical studies on the efficiency loss 
of equilibria in the DSPI game compared to optimal
solutions obtained through centralized decision making.

Decentralized network interdiction games, as will be formally defined in Section
2, appear to be new. To the best of our knowledge, there has been no previous
research on such games. As a result, not much is known about the inefficiency of
equilibria for these games or intervention strategies to reduce such
inefficiencies. There has been a considerable amount of work, however, on
interdiction problems from a centralized decision-maker's perspective. As
mentioned earlier, interdiction problems have been studied in the context of
various military and security applications.  For extensive reviews of the
existing academic literature on interdiction problems, we refer the readers to
Church et al.  \cite{church_crit_infra_04} and Smith and Lim
\cite{smith_liu_algo_2008}.

There have also been many studies on the inefficiency of equilibria in other
game-theoretic settings. Most of the efforts have been focused on routing games
\cite{pigou1924economics, wardrop1952road}, in which
selfish agents route traffic through a congested network, and congestion games
\cite{rosenthal1973class}, a generalization of routing games. Some examples
include \cite{roughgarden2002bad, correa2004selfish,
awerbuch2005price, christodoulou2005price1,
cole2006much, suri2007selfish}. Several researchers
have also studied the inefficiency of equilibria in network formation games, in
which agents form a network subject to potentially conflicting connectivity
goals \cite{albers2006nash, anshelevich2008price, 
fabrikant2003network}. The inefficiency of equilibria has been
studied in other games as well, such as facility location games
\cite{vetta2002nash}, scheduling games \cite{koutsoupias1999worst}, and resource
allocation games \cite{johari2004efficiency, johari2009efficiency}. Almost all
of the work described above study the worst-case inefficiency of a given
equilibrium concept.  Although a few researchers have studied the average
inefficiency of equilibria, either theoretically or empirically, and have used
it as a basis to design interventions to reduce the inefficiency of equilibria
\cite{corbo2005price, thompson2009computational}, research in this
direction has not received much attention. 

One potential reason for the lack of attention paid to decentralized network
interdiction games may be that such games  often involve nondifferentiability,
as each interdictor's optimization problem usually entails a max-min type of
objective functions. Games involving nondifferentiable functions are generally
challenging, in terms of both theoretical analysis of their equilibria and
computing an equilibrium.  While in some cases (such as in the case of shortest
path interdiction), a smooth formulation (through total unimodularity and
duality) is possible, such a reformulation will lead the resulting network game
to the class of generalized Nash equilibrium problems (GNEPs), in which both the
agents' objective functions as well as their feasible action spaces depend on
other agents' actions. Although the conceptual framework of GNEPs can be dated
to Debreu \cite{debreu1952social}, rigorous theoretical and algorithmic
treatments of GNEPs only began in recent years \cite{gnep_survey}.  Several
techniques have been proposed to solve GNEPs, including penalty-based approaches
\cite{doi:10.1137/090749499, fukushima2011restricted},
variational-inequality-based approaches \cite{nabetani2011parametrized},
Newton's method \cite{dreves4}, projection methods \cite{zhang2010some}, and
relaxation approaches \cite{krawczyk2000relaxation1, uryas1994relaxation}. Most
of the work on GNEPs has focused on games with shared constraints due to their
tractability \cite{Facchinei2007159, Harker199181}. In such games, a set of
identical constraints appear in each agent's feasible action set. However, as
will be seen later, in a typical decentralized network interdiction game, the
constraints involving multiple agents' actions that appear in each agent's
action space are not identical. As a result, such games give rise to more
challenging instances of GNEPs. 

Based on the discussions above, the major contributions of this work are as
follows. First, we establish the existence of equilibria for DSPI games with
continuous interdiction. In DSPI games with discrete interdiction, the existence
of a pure strategy Nash equilibrium (PNE) is more subtle. We first demonstrate
that a PNE does not necessarily exist in general discrete DSPI games. However,
when all agents have the same source-target pairs (i.e., multiple agents try to
achieve a common goal independently), a PNE exists in discrete DSPI games.
Second, for DSPI games under continuous interdiction, we show that each agent's
optimization problem can be reformulated as a linear programming problem. As a
result, the equilibrium conditions of the game can be reformulated as a linear
complementarity problem with some favorable properties, allowing it to be solved
by the well-known Lemke algorithm.  For discrete DSPI games (and for continuous
games as well),  we present decentralized algorithms for finding an equilibrium,
based on the well-known best-response dynamics (or Gauss-Seidel iterative)
approach.  While such an approach is only a heuristic method in general,
convergence can be established for the special case when the agents have common
source-target pairs. For more general cases, we obtain encouraging empirical
results for the performance of the method on several classes of network
structures.  Third, in measuring the efficiency loss of DSPI games due to the
lack of coordination among noncooperative interdictors, as compared to a
centralized interdiction strategy (that is, a strategy implemented by a single
interdictor with respect to all the adversaries), we establish a theoretical
lower bound for the worst-case price of anarchy of DSPI games under continuous
interdiction. Such an efficiency loss measure, however, may be too conservative,
and we therefore use the decentralized algorithms to empirically quantify the
\emph{average-case} efficiency loss over some instances of DSPI games. These
results can help central authorities design mechanisms to reduce such efficiency
losses for practical instances.  

The remainder of this paper is organized as follows. We begin in Section 2 with
definitions and formulations of DNI games and DSPI games. In Section 3, we
present the main theoretical results of the paper, including an analysis of the
existence and uniqueness of equilibria in DSPI games. In Section 4 we
investigate algorithms for solving DSPI games. We describe a centralized
algorithm based on a linear complementarity formulation, as well as
decentralized algorithms for computing equilibria of DSPI games. We also give
results of our computational experiments with these algorithms for computing
equilibria as well as quantifying the price of anarchy for various instances.
Finally, in Section 5, we provide some concluding remarks.

\section{Decentralized Network Interdiction Games}


\subsection{Formulation}
\label{formulation}

Network interdiction problems involve interactions between two types of
parties -- adversaries and interdictors -- with conflicting interests.
An adversary operates on a network and attempts to optimize some
objective, such as the flow between two nodes. An interdictor tries to
limit an adversary's objective by changing elements of the network, such
as the arc capacities. Such interactions have historically been viewed from
a leader-follower-game perspective. The interdictor acts as the leader
and chooses an action while anticipating the adversary's potential
responses, while the adversary acts as the follower and makes a move
after observing the interdictor's actions. From the interdictor's
perspective, this captures the pessimistic viewpoint of guarding against
the worst possible result given its actions.

In this work, we consider strategic interactions among multiple
interdictors who operate on a common network. The interdictors may each
have their own adversary or have a common adversary. If there are
multiple adversaries, we assume there is no strategic interaction among
them. We also assume that the interdictors are allies in the sense that they
are not interested in deliberately impeding each other. 

Formally, we have a set $\mathcal{F} = \{1,\ldots,F\}$ of interdictors or
agents, who operate on a network $G=(V,A)$, where $V$ is the set of nodes and
$A$ is the set of arcs. Each agent's actions or decisions correspond to
interdicting each arc of the network with varying intensity: the decision
variables of agent $f \in \mathcal{F}$ are denoted by $x^f \in X^f\subset
\Real^{\vert A \vert}$, where $X^f$ is an abstract set that constrains agent
$f$'s decisions.  
For any agent $f \in \mathcal{F}$, let $x^{-f}$ denote the collection of all the other
agents' decision variables; that is, $x^{-f} = (x^1, \ldots, x^{f-1}, x^{f+1},
\ldots, x^F)$. The network obtained after every agent executes its decisions or
interdiction strategies is called the \emph{aftermath network}. The strategic
interaction between the agents occurs due to the fact that the properties of
each arc in the aftermath network are affected by the combined decisions of all
the agents.

In addition to the abstract constraint set $X^f$, we assume that each agent $f
\in \mathcal{F}$ faces a total interdiction budget of \textbf{$b^f$}.
The cost of interdicting an arc is linear in the intensity of
interdiction; in particular, agent~$f$'s cost of interdicting arc $(u,v)$ by
$x^f_{uv}$ units is $c^f_{uv} x^f_{uv}$.  Without loss of generality, 
we assume that $b^f > 0$ and
$c^f_{uv} > 0$ for each arc $(u,v) \in A$ and for each agent $f \in \mathcal{F}$.

\noindent The optimization problem for each agent $f \in \mathcal{F}$ 
is: 
\begin{equation}\label{dni}
  \begin{aligned}
    \maximize{x^f} \quad & \theta^f(x^f, x^{-f}) \\
    \sbt \quad & \sum_{(u,v) \in A} c^f_{uv} x^f_{uv} \leq b^f, \\
    & x^f \in X^f,
  \end{aligned}
\end{equation}
where the objective function $\theta^f$ is agent $f$'s \emph{obstruction
  function}, or measure of how much agent~$f$'s adversary has been
obstructed. Henceforth, we refer to the game in which each agent $f\in
\mathcal{F}$ solves the above optimization problem \eqref{dni} as a
\emph{decentralized network interdiction (DNI) game}. The obstruction
function~$\theta^f$ can capture various types of interdiction problems.
Typically $\theta^f$ is the (implicit) optimal value function of the
adversary's network optimization problem parametrized by the
agents' decisions, which usually minimizes flow cost or path length
subject to flow conservation, arc capacity and side constraints. 

Suppose that a central planner, with a comprehensive view of the network
and the agents' objectives, could pool the agents' interdiction
resources and determine an interdiction strategy that maximizes some
global measure of how much the agents' adversaries have been obstructed.
Let $\bm{\theta^c}(x^1, \ldots, x^F)$ represent the global obstruction
function for a given interdiction strategy $(x^1, \ldots, x^F)$. 
The central planner's problem corresponding to the DNI game \eqref{dni}
is then:
\begin{equation}\label{dni_centr_1}
  \begin{aligned}
    \maximize{x^1,\ \ldots,\ x^F} \quad & \bm{\theta^c}(x^1, \ldots, x^F) \\
    \sbt \quad & \sum_{f \in \mathcal{F}} \sum_{(u,v) \in A} c^f_{uv} x^f_{uv} \leq \sum_{f \in \mathcal{F}} b^f, \\
    &x^f \in X^f \quad \forall f \in \mathcal{F}.
  \end{aligned}
\end{equation}
We refer to (\ref{dni_centr_1}) as the centralized problem, and focus
primarily on when the global obstruction function is \emph{utilitarian};
that is, 
\begin{equation*} 
  \bm{\theta^c}(x^1, \ldots, x^F) := \sum_{f \in \mathcal{F}} \theta^f(x^f, x^{-f}). 
\end{equation*}
As mentioned earlier, one of the goals of this work is to quantify the
inefficiency of an equilibrium of a DNI game -- a decentralized solution
to problem~\eqref{dni} -- relative to a centrally planned optimal
solution -- an optimal solution to problem~\eqref{dni_centr_1}. A
commonly used measure of such inefficiency is the \emph{price of
  anarchy}. 
Formally speaking,  
let $\mathcal{N}_I$ be the set of all equilibria corresponding
to a specific instance $I$. (In the context of DNI
games, an instance consists of the network, obstruction functions,
interdiction budgets, and costs). 
For the same instance $I$,
let $(x^{1^*}, \ldots, x^{F^*})$ denote a global optimal solution to the
centralized problem \eqref{dni_centr_1}. Then the price of anarchy of the  instance $I$ is defined as 
\begin{equation}\label{PoA}
  p(I) := \maximi{(x^1_N, \ldots, x^F_N) \in \mathcal{N}_I} \frac{\bm{\theta^c}(x^{1^*}, \ldots x^{F^*})}{\bm{\theta^c}(x^1_N, \ldots, x^F_N)}.
\end{equation}
Let $\mathcal{I}$ be the set of all instances of a game. We assume implicitly that for all $I\in\mathcal{I}$, the set
$\mathcal{N}_I$ is nonempty and a global optimal solution to the
centralized problem exists. By convention, $p$ is set to 1 if the
worst equilibrium as well as the global optimal solution to the
centralized problem both have zero objective value. If the worst
equilibrium has a zero objective value while the global optimal value of
the centralized problem is nonzero, $p$ is set to be infinity. 
In addition to the price of anarchy for an instance of a game, 
we also define the worst-case price of anarchy over all instances of the game (denoted as $w.p.o.a$)  as follows:
\begin{equation}\label{wpoa}
  w.p.o.a:= \sup_{I \in \mathcal{I}} \  p(I). 
\end{equation}
Since we wish to study the properties of a class of games
  such as DNI games, rather than a particular instance of a game, we
  are more interested in the worst-case price of anarchy.   However,
  there are two major difficulties associated with such an efficiency
  measure.
First,  it is well-known that the worst-case price of anarchy may be a
very conservative measure of efficiency loss, since the worst case may
only happen with pathological instances. 
Second, explicit theoretical bounds on the worst-case price of anarchy
may be difficult to obtain for general classes of games. In fact most of
the related research has focused on identifying classes
of games where such bounds may be derived. In this work, we show how our proposed decentralized algorithms can be used to empirically study the \emph{average-case efficiency loss} (denoted by $a.e.l$). Let $\mathcal{I}'$ denote a finite set such that $\mathcal{I}' \subset \mathcal{I}$, and let $\vert \mathcal{I}' \vert$ denote the cardinality of the the set $\mathcal{I}'$. Then    
\vspace*{-3pt}
\begin{equation} 
  a.e.l(\mathcal{I}') := \displaystyle \frac{1}{\vert \mathcal{I}' \vert} \displaystyle \sum_{I \in \mathcal{I}'} \  p(I).
\end{equation}

\noindent In other words, the average-case efficiency loss is the average value of
$p(I)$ as defined in \eqref{PoA} over a set of sampled instances
$\mathcal{I}' \subset \mathcal{I}$ of a game.

As mentioned above, the generic form of problem \eqref{dni} can be used to
describe various network interdiction settings, such as maximum flow
interdiction. To start with models that are both theoretically and
computationally tractable, we focus on decentralized shortest-path interdiction
games, which we describe in detail next.

\subsubsection{Decentralized Shortest Path Interdiction Games}
\label{dspi}

As the name suggests, \emph{decentralized shortest path interdiction
  (DSPI)} games involve agents or interdictors whose adversaries 
are interested in the shortest path between source-target node pairs on
a network. Interdictors act in advance to increase the length of the
shortest path of their respective adversaries by interdicting (in
particular, lengthening) arcs on the network. 

To describe these games formally, we build upon the setup of the
general decentralized network interdiction game described in
Section~\ref{formulation}. Each agent $f \in \mathcal{F}$ has a target
node $t^f \in V$ that it wishes to protect from an adversary at source
node $s^f \in V$ by maximizing the length of the shortest path between
the two nodes. The agents achieve this goal by committing some resources
(e.g.~monetary spending) to increase the individual arc lengths on the
network: the decision variable $x^f_{uv}$ represents the contribution of
agent $f \in \mathcal{F}$ towards lengthening arc $(u,v) \in A$. The arc
length $d_{uv}(x^f, x^{-f})$ of arc $(u,v) \in A$ in the aftermath
network depends on the decisions of all the agents. 

We consider two types of interdiction. The first type of interdiction is
\emph{continuous}: in particular,
\begin{equation*}
  X^f := \{ x^f \in \mathbb{R}^{|A|} : x^f_{uv} \ge 0 \quad \forall (u,v) \in A
  \}.
\end{equation*}
The arc lengths after an interdiction strategy $(x^1, \ldots, x^F)$
has been executed are
\begin{equation}\label{ArcDistCont}
  d_{uv}(x^1, \ldots, x^F) = d_{uv}^0 + \sum_{f \in \mathcal{F}} x^f_{uv} \quad
  \forall (u,v) \in A,
\end{equation}\vspace*{-10pt}

\noindent where $x^f_{uv}$ captures how much agent $f$ extends the
length of arc~$(u,v)$. We assume that $d^0_{uv} > 0$
for all $(u,v) \in A$. 

The second type of interdiction is
\emph{discrete}: in this case, 
\vspace*{-5pt}
\begin{equation*}
  X^f := \{ x^f \in \mathbb{R}^{|A|} : x^f_{uv} \in  \{0,1\} \quad \forall (u,v) \in A \}
\end{equation*}
and the arc lengths in the aftermath network are
\begin{equation} \label{ArcDistDisc}
  d_{uv}(x^1, \ldots, x^F) = d_{uv}^0 + e_{uv} \maximi{f \in \mathcal{F}} x^f_{uv} \quad \forall (u,v) \in A,
\end{equation}
where $e_{uv} \in \Real_{\ge 0}$ is the fixed extension of arc $(u,v)$.
In other words, the length of an arc is extended by a fixed amount if at
least one agent decides to interdict it.

Let $P^f = \{p^f_1,p^f_2,\ldots,p^f_{k^f}\}$ be the set of $s^f-t^f$ paths available to agent $f\in\mathcal{F}$. The length of a path $p \in P^f$ is given by 
\begin{equation}\label{d_p}
d_p(x^1,\ldots,x^F) = \sum_{(u,v) \in p} d_{uv} (x^1,\ldots,x^F),
\end{equation}
where $d_{uv} (x^1,\ldots,x^F)$ is as defined in equation \eqref{ArcDistCont} for continuous interdiction, and as defined in \eqref{ArcDistDisc} for the discrete case.

The optimization problem for each interdicting agent $f \in \mathcal{F}$ is then: 
\begin{equation}\label{dni_paths}
  \begin{aligned}
    \maximize{x^f} \quad & \theta^f(x^f, x^{-f}) \equiv \min_{p \in P_f} d_p(x^f,x^{-f})  \\
    \text{subject to} \quad & \sum_{(u,v) \in A} c^f_{uv} x^f_{uv} \leq b^f, \\
    & x^f \in X^f.
  \end{aligned}
\end{equation}

Under continuous interdiction and the general assumption made ealier that $X_f$
is nonempty, convex and compact, the feasible strategy set for agent $f$, given
by $\{x^f \in X^f | \sum_{(u,v) \in A} c^f_{uv} x^f_{uv} \leq b^f \}$ is also
convex and compact. To rule out uninteresting cases, we also assume that the
feasible set for each agent is also nonempty (meaning that each agent has the
budget to at least interdict one arc). Moreover, given an $x^{-f}$, the
objective function in \eqref{dni_paths} is the minimum of a set of affine
functions of $x^{f}$, and therefore continuous in $x^{f}$. Thus, by Weirstrass's
extreme value theorem, each agent has an optimal strategy given the strategies
of the other agents.  Note, however, that the objective function in
\eqref{dni_paths} is not differentiable with respect to $x_f$ in general. 

For DSPI games with discrete interdiction, the feasible strategy set for each
agent is finite. Therefore an optimal solution to each agent's problem always
exists with a given $x^{-f}$. In the following section, we analyze the existence
and uniqueness of pure strategy Nash Equilibria for DSPI games, under both
continuous and discrete settings.

\section{Game Structure and Analysis}

\subsection{Existence of Equilibria}

We first consider the existence of a Nash equilibrium  of a DSPI game when interdiction decisions are continuous. The key is to show that the objective function in \eqref{dni_paths}, $\theta^f(x^f, x^{-f})$, is concave in $x^f$, despite the fact that it is not differentiable.
\begin{proposition}
Given that each agent $f\in\mathcal{F}$ solves the problem \eqref{dni_paths}, with $d_p(x^f,x^{-f}) $ defined as in \eqref{d_p} and \eqref{ArcDistCont}, and assume that the abstract set $X_f$ in \eqref{dni_paths} is nonempty, convex and compact for each $f \in \mathcal{F}$, 
the DSPI game under continuous interdiction has a pure strategy Nash equilibrium.
\end{proposition}
\begin{proof}
Based on the assumption, the feasible region in  \eqref{dni_paths} is nonempty, convex and compact.
With a fixed $x^{-f}$, the objective function of agent $f$ is the minimum of a finite set  of affine functions in $x^f$, and therefore, is concave with respect to $x^f$, by the well-known fact in convex analysis (Cf. \cite{boyd2004convex}). Consequently, the DSPI game belongs to the class of ``concave games," introduced in Rosen \cite{rosen1965existence}, and it is shown in \cite{rosen1965existence} that a pure-strategy Nash equilibrium always exists for a concave game. 
\end{proof}

Under discrete interdiction, the existence of a PNE is not always guaranteed when different interdictors are competing against different adversaries. We illustrate the nonexistence of PNE in Example \ref{example0} below.

\begin{example} \label{example0}
\rm{Consider the network given in Figure \ref{ex0}.}
\begin{figure}[htp]
\centering
\includegraphics[width=0.5\textwidth]{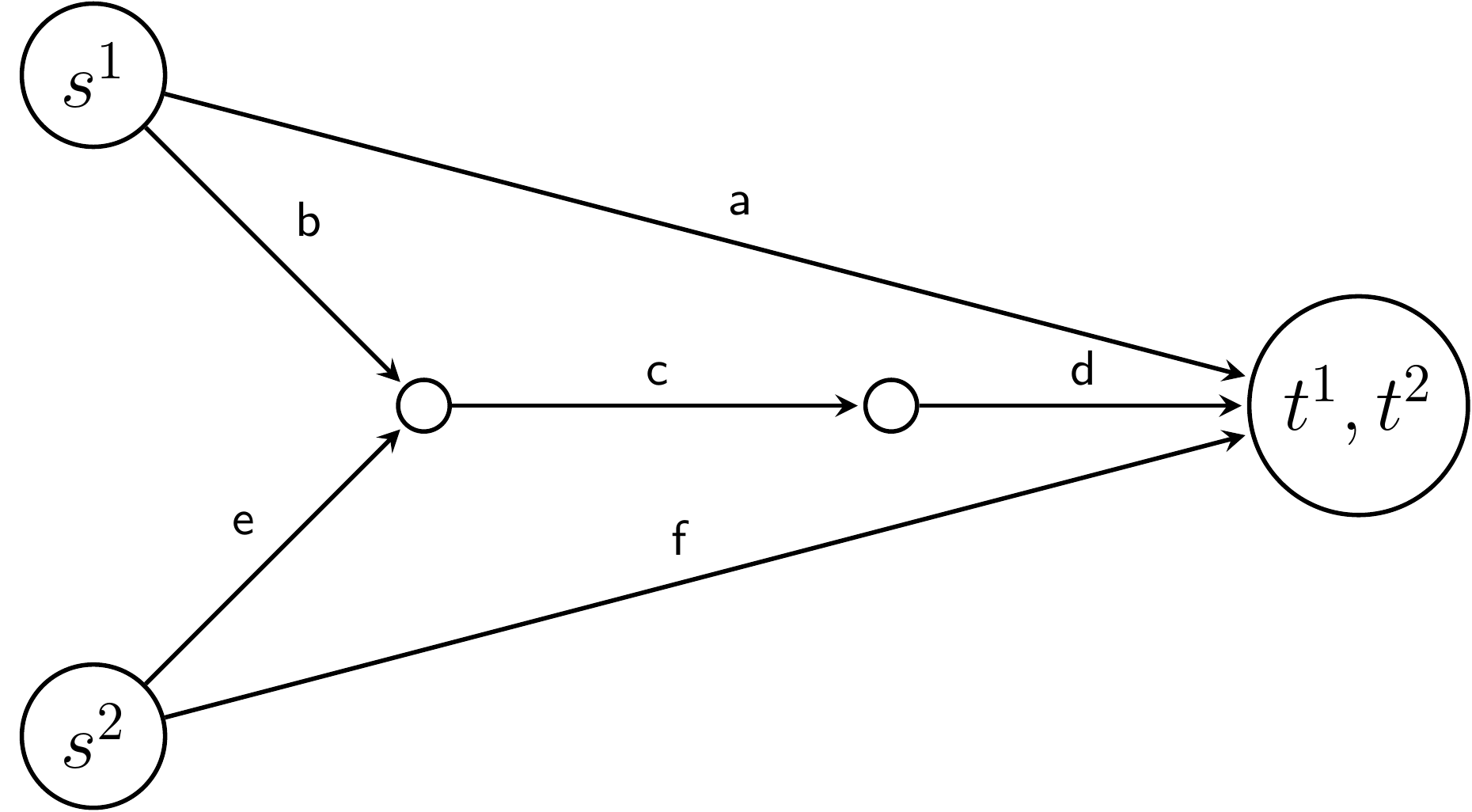}
\caption{Network topology for DSPI game in Example \ref{example0}.}
\label{ex0}
\end{figure}

\rm{In this game, there are two agents -- agent 1 and agent 2 -- who
are 
attempting to maximize the lengths of the $s^1$-$t^1$ paths and $s^2$-$t^2$ paths respectively.
Note that $t^1 = t^2$. The data for the problem, including initial arc lengths,
cost of interdiction and arc extensions are given below in Table
\ref{table:properties}.}

\begin{table}[h]
\centering{%
\begin{tabular}{| c | c | c | c | c |}
\hline
Arc tag & Initial length &  Arc extension & Cost to player $1$ & Cost to player $2$  \\\hline
a & 7 & 0.5 & 3 & 20  \\\hline
b & 0 & 2 & 6 & 20  \\\hline
c & 0 & 1.5 & 5 & 20  \\\hline
d & 0 & 6 & 15 & 15 \\\hline
e & 0 & 1 & 20 & 20  \\\hline
f & 1 & 6 & 15 & 15  \\\hline
\end{tabular}}
\caption{Network data for Example 1}
\label{table:properties}
\end{table}

\rm{Suppose $b^1 = 8$ and $b^2 = 15$. As a result, player $1$ can either
interdict the arcs $a,b$ and $c$ one at a time, or the arcs $a$ and $c$
simultaneously. Similarly, player $2$ can either interdict arc $d$ or arc $f$. }

\rm{Thus, player $1$ has four feasible pure strategies and player $2$ has two
feasible pure strategies. The strategy tuples along with the corresponding
pay-offs for each player are summarized in Table \ref{table:discretepotential}.
It is easy to verify that for any joint strategy profile, there is a player who
would prefer to deviate unilaterally. Therefore, this instance of the DSPI game
does not possess a NE.  }

\begin{table}[h]
\centering{%
\begin{tabular}{| c | c | c | }
\hline 
$P_1$/$P_2$ strategies & $d$ & $f$   \\\hline
$a$ & $6,1$ & $0,0$ \\\hline
$c$ & $7,1$ & $1.5,1.6$ \\\hline
$(a,c)$ & $7.5,1$ & $1.5,1.5$ \\\hline
$b$ & $7,1$ & $2,0$ \\\hline
\end{tabular}}
\caption{Pay-off combinations for Example \ref{example0}}
\label{table:discretepotential}
\end{table}%

\end{example}

In the previous example, the agents have a common target node, but different source nodes. However, in the class of games in which the interdictors have a common adversary, i.e., when each agent maximizes the shortest path between a common source-target pair, we can show that DSPI games under discrete interdiction possess a PNE.

Consider the DSPI game where each agent is trying to maximize the shortest path lengths between nodes $s$ and $t$. Since the objective function of each agent is the same, we can write the following centralized optimization problem to maximize the shortest $s-t$ path distance subject to the individual agents' budget constraints. Let $P^{st}$ be the set of $s-t$ paths in the network. The centralized optimization problem is: 
\begin{equation}\label{discrete_commonpath}
  \begin{aligned}
    \maximize{x} \quad & \min_{p \in P^{st}} d_p(x^1,x^2,\ldots,x^F)  \\
    \text{subject to} \quad & \sum_{(u,v) \in A} c^f_{uv} x^f_{uv} \leq b^f \quad \forall f \in \mathcal{F}, \\
    & x^f_{uv} \in \{0,1\} \quad \forall (u,v) \in A, f \in \mathcal{F}.
  \end{aligned}
\end{equation}
The feasible solution space of the above problem is finite under individual agents' budget constraints. Therefore, the centralized problem always has a maximum. Furthermore, the optimal solution to this problem is a PNE of the DSPI game as we show in the following result. 

\begin{proposition} \label{NE_existence_commonpaths}
Suppose the source and target for each agent in a DSPI problem are the same. Let $x^*$ denote the optimal solution of the centralized problem \eqref{discrete_commonpath}. Then $x^*$ is a PNE to the DSPI game under discrete interdiction.
\end{proposition}
\begin{proof}
Assume the contrary, and suppose that there is an agent $h$ for whom there exists a
unilateral deviation $x^f$ that strictly increases the path distance $s-t$. By assumption, $x^h$ is feasible to the budgetary constraints for agent $h$. Therefore, $\bar{x} \equiv (x^h,\ {x^*}^{-h})$ is feasible to \eqref{discrete_commonpath} with a strictly larger objective value. Clearly this is a contradiction to the optimality of $x^*$ to \eqref{discrete_commonpath}.
\end{proof}

\subsubsection{Uniqueness of equilibria}
\label{uniqueness}

Establishing sufficient conditions for a DSPI game to have a unique equilibrium
is quite difficult. However, it is easy to find simple
instances of DSPI games for which multiple equilibria exist. 
We give two such examples below.

\begin{example} \label{example1}
\rm{Consider the following instance, based on the network in Figure
\ref{fig1}. There are 2 agents: agent~1 has an adversary with source
node 1 and target node 5; agent~2 has an adversary with source node
1 and target node 6. The initial arc lengths are 0, interdiction is
continuous, and the interdiction costs are the same for both agents and
are given in the arc labels in Figure~\ref{fig1}. Both agents have a
budget of 1.
\begin{figure}[htp]
\centering
\includegraphics[width=0.5\textwidth]{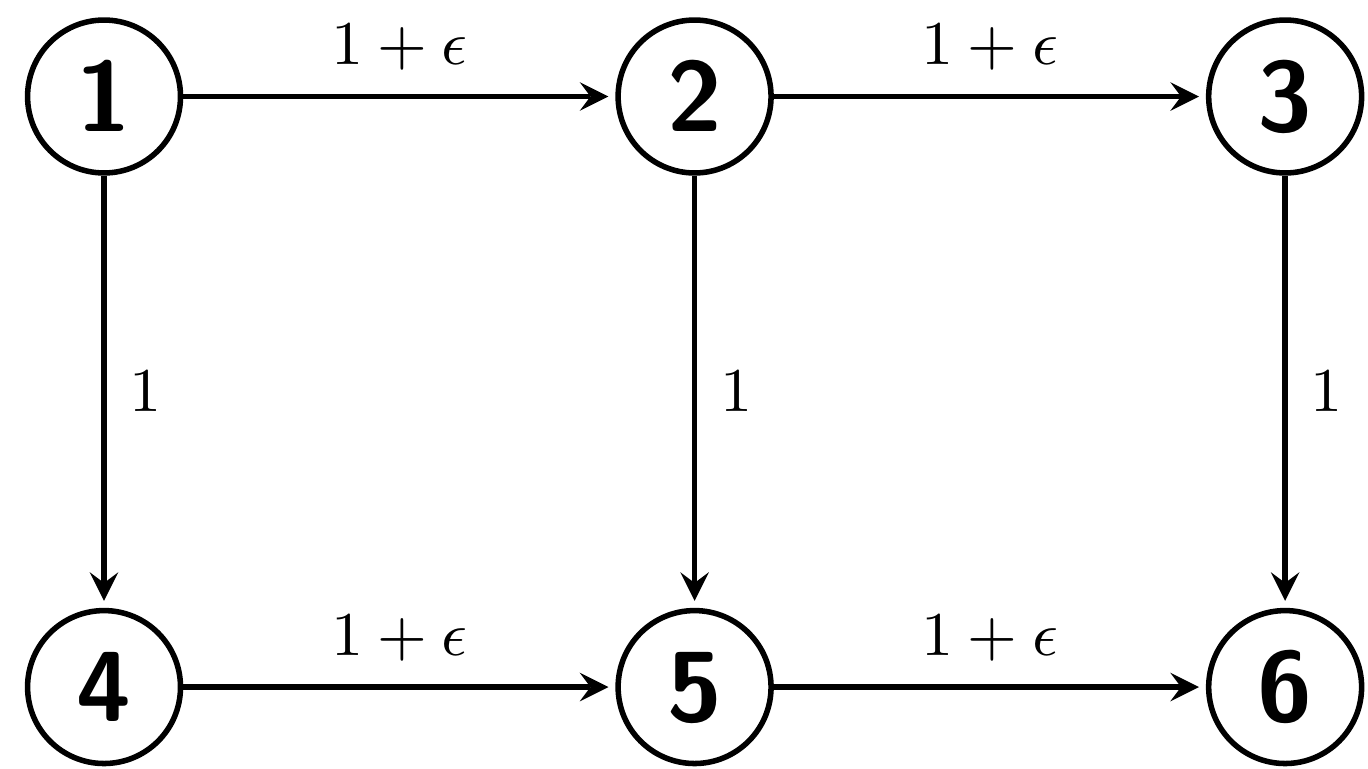}
\caption{Network topology for DSPI game in Example \ref{example1}.}
\label{fig1}
\end{figure}
Consider the case when $\epsilon = 2$.  In this case, it is straightforward to
see that the source-target path lengths for each agent must be equal at an
equilibrium: if the path lengths are unequal, an agent could improve its
objective function by equalizing the path lengths.  Therefore, in this example,
any combination of decision variables that results in a shortest path length of
$2/3$ for each agent will be a generalized Nash equilibrium, and there is a
continuum of such decision variable combinations. Indeed some of such equilibria
are given in Table~\ref{multi_eq_table} in Section~\ref{sec:numerical}. }
\end{example}

\begin{example} \label{example3} 
\rm{Under discrete interdiction on the same underlying network, an interesting
situation occurs when $\epsilon = 0$, the budget is 1, and the arc extensions
are all set to 1. In this case, 
an equilibrium occurs when the arcs $(1,4)$ and $(1,2)$ are interdicted by one agenteach. 
What is interesting however is that there exist equilibria that have
inferior objective values for both agents. Indeed, the extreme case of neither agent 
interdicting any arc can easily seen to be an equilibrium. This point in fact is
a social utility \emph{minimizer} over the set of feasible action combinations for
the two agents.}
\end{example}

\section{Computing a Nash Equilibrium}

In this section we discuss algorithms to compute equilibria of DSPI games. While
the general formulation with each agent solving \eqref{dni_paths} is sufficient
for showing existence of equilibria, such a formulation is not amenable for
computing an equilibrium mainly due to the `$\min$' function in the objective
function. In this section, using a well-known reformulation of shortest path
problems (through total unimodularity and linear programming duality), we
obtain a generalized Nash equilibrium problem of the DSPI game.  For
continuous DSPI games, we further show that such a GNEP can be written as a
linear complementarity problem (LCP) through the Karush-Kuhn-Tucker (KKT)
optimality conditions. We then show that the resulting LCP has favorable
properties, allowing the use of Lemke's pivoting algorithm with guaranteed
convergence to a solution (as opposed to a secondary ray).

We refer to the LCP approach as a centralized approach, in the sense that the game is purely viewed as a system of equilibrium conditions, and a general algorithm capable of solving the resulting system is applied. 
We also present decentralized algorithms based on
best-response dynamics, which are applicable to both continuous and discrete DSPI games. 
While not necessarily computationally more efficient, decentralized algorithms indeed have several
advantages over centralized algorithms. First, a centralized algorithm usually cannot solve discrete games, which can nevertheless be solved by a decentralized approach, aided by integer programming solvers. 
Second, a decentralized algorithm may provide
insight on how a particular equilibrium is achieved among agents' strategic
interactions. Such insight is particularly useful when multiple
equilibria exist, as is the case for many GNEPs. It is well-known (for
example, \cite{Myerson78}) that a game may possess unintuitive Nash
equilibria that would never realistically be the outcome of the game. 
Third, a decentralized algorithm can naturally lead to
multithreaded implementations that can take advantage of a high
performance computing environment. In addition, different threads in a
multithreaded implementation may be able to find different equilibria of
a game, making such an algorithm particularly suitable for
computationally quantifying the average efficiency loss of noncooperative strategies. 

In the following discussion, we first present the GNEP formulation of DSPI games under continuous interdiction. We then reformulate the GNEP as an LCP and analyze the properties of the LCP formulation. Finally we present the decentralized algorithms for both discrete and continuous DSPI games formulated as GNEPs.

\subsection{Dual GNEP formulation} 
In \eqref{dni_paths}, each agent's objective function $\theta^f(x^f, x^{-f})$ involves its adversary's shortest path problem, which can be written as an integer program as follows:
\begin{equation}\label{spi_primal}
  \theta^f(x^f, x^{-f}) = \left(
    \begin{aligned}
      \minimi{z^f} \quad & \sum_{(u,v) \in A} z^f_{uv} \  d_{uv}(x^f, x^{-f}) & & \\
      \text{s.t.} \quad & \sum_{v \in V} z^f_{uv} - \sum_{v \in V} z^f_{vu} = 
      \begin{cases}
        1 &\mbox{if } u = s^f \\
        0 & \mbox{if } u \neq s^f, t^f \\
        -1 & \mbox{if } u = t^f 
      \end{cases} \\
      & z^f_{uv} \in \{ 0,1 \} \quad \forall (u,v) \in A \\
    \end{aligned} \right),
\end{equation}
where the binary variables $z^f_{uv}$ represents
whether an arc $(u,v) \in A$ is in the shortest $s^f$-$t^f$ path. 
Although the inner minimization problem is an integer program with
binary variables, it is well known that the constraint matrix is totally
unimodular (e.g. \cite{schrijver1998theory}), rendering the integer
program equivalent to its linear programming relaxation. Therefore, once
the interdictors' variables $(x^1,\dots,x^F)$ are fixed, we can use
linear programming duality to transform the inner minimization problem
to a maximization problem\cite{israeli_wood_02} and reformulate
agent~$f$'s optimization problem~\eqref{dni} as: 
%
%
\begin{equation} \label{dualpotential3}
  \begin{aligned}
    \maximize{x^f,\ y^f} \quad & y^f_{t^f} - y^f_{s^f} \\
    \text{subject to} \quad & y^f_{v} - y^f_{u} \leq d_{uv}(x^f, x^{-f}) \quad \forall (u,v) \in A, \\[5pt]
    & \sum_{(u,v) \in A} c^f_{uv} x^f_{uv} \leq b^f, \\
    & x^f \in X^f, \\
    & y^f_v \geq 0 \quad \forall v \in V.
  \end{aligned}
\end{equation}
%
%
It is also well-known (see, for example, \cite{bazaraa2011linear,
  korte2012combinatorial}) that at optimality, the term $y_u^f - y^f_{s^f}$
is equal to the length of the shortest $s^f$-$u$ path in the aftermath
network. This is the reason why we are able to restrict the $y^f$ variables to
be non-negative. In addition, it also allows us to restrict the $y^f$
variables to be integral if the underlying network data is integral,
since at optimality all path lengths would also be integral. Moreover,
as we show below, it also allows us to bound the $y^f$ variables.

When interdiction is continuous, the largest possible length in the
aftermath network for any arc is bounded by the largest interdiction
possible on that arc. Keeping the budgetary constraints in mind, the maximum length of any arc $(u,v) \in A$ in the aftermath
network is bounded by 
\vspace*{-5pt}
\begin{equation*}
  d_{uv}^0 +  F \cdot \maximi{f \in \mathcal{F},\  (u,v) \in A}\displaystyle \left\{ \frac{b^f}{c^f_{uv}}\right\}.
\end{equation*}
Therefore, the length of any path in the aftermath network is bounded above by
\begin{equation*}\vspace*{-2pt}
  M = \sum_{(u,v) \in A} d^0_{uv} + \vert A \vert \ F \cdot \maximi{f \in \mathcal{F},\ a \in A}\displaystyle \left\{ \frac{b^f}{c^f_a}\right\}.
\end{equation*}\vspace*{-5pt}

\noindent On the other hand, when interdiction is discrete, the length of any
path in the aftermath network is bounded above by
$  M = \sum_{(u,v) \in A} (d^0_{uv} + e_{uv}).$
%

Since only the differences $y^f_v - y^f_u$ across arcs $(u,v)$ are
relevant to the formulation \eqref{dualpotential3}, we may always
replace $y^f_u$ by $y^f_u - y^f_{s^f}$ for each $u \in V$ to
obtain a feasible solution with equal objective value. 
Therefore we can then add the constraints $0 \leq y^f_u  \leq
M$ for all
$u \in V$ to the problem \eqref{dualpotential3} to obtain an equivalent
formulation of a DSPI game, where each agent~$f \in \mathcal{F}$ solves
the following problem:
\vspace*{-3pt}
\begin{equation} \label{dualpotential}
  \begin{aligned}
    \maximize{x^f,\ y^f} \quad & y^f_{t^f} - y^f_{s^f} \\
    \text{subject to} \quad & y^f_{v} - y^f_{u} \leq d_{uv}(x^f, x^{-f}) \quad \forall (u,v) \in A, \\
    & \sum_{(u,v) \in A} c^f_{uv} x^f_{uv} \leq b^f, \\
    & 0 \leq y^f_u \leq M \quad \forall u \in V, \\
    & x^f \in X^f.
  \end{aligned}
\end{equation}\vspace*{-3pt}

When analyzing the DSPI game from a centralized decision-making
perspective, we assume that the global obstruction function is
utilitarian, i.e., the sum of the shortest $s^f$-$t^f$ path lengths over
all the agents~$f \in \mathcal{F}$.  We also assume that the resources
are pooled among all the agents, resulting in a common budgetary
constraint. Thus the centralized problem for DSPI games can be given as follows:
\begin{equation} \label{dualcentrdspi}
  \begin{aligned}
    \maximize{x,\ y} \quad & \sum_{f \in \mathcal{F}} \big( y^f_{t^f} - y^f_{s^f} \big) \\
    \text{subject to} \quad & y^f_{v} - y^f_{u} \leq d_{uv}(x^f, x^{-f}) \quad \forall (u,v) \in A, f \in \mathcal{F}, \\
    & \sum_{f \in \mathcal{F}} \sum_{(u,v) \in A} c^f_{uv} x^f_{uv} \leq \sum_{f \in \mathcal{F}} b^f \\
    & 0 \leq y^f_u \leq M \quad \forall u \in V, f \in \mathcal{F}, \\
    & x^f \in X^f \quad \forall f \in \mathcal{F}.
  \end{aligned}
\end{equation}
Since $y^f$ is bounded for all $f \in \mathcal{F}$, the feasible set for
\eqref{dualcentrdspi} is a compact set. Thus a globally optimal solution exists
regardless of whether $x^f$ is continuous or discrete for all $f \in
\mathcal{F}$. In the continuous case, Weierstrass's extreme value theorem
applies since all the functions are continuous and the $x^f$ variables are
bounded due to the non-negativity and budgetary constraints.  In the discrete
case, there are only a finite number of values that the $x^f$ variables can
take. 

%

The formulation \eqref{dualpotential} gives us some insight into the
structure of strategic interactions among agents in a DSPI game. Note
that in formulation \eqref{dualpotential}, the objective function for
each agent $f \in \mathcal{F}$ only depends on variables indexed by $f$
(in particular, $y^f_{s^f}$ and $y^f_{t^f}$). However, the constraint
set for each agent $f$ is parametrized by other agents' variables
$x^{-f}$, which leads to a \emph{generalized Nash equilibrium
problem}. 

Formally speaking, consider a simultaneous-move game with
complete information.\footnote{A game is said to be
  \emph{simultaneous-move} if the agents must make their decisions
  without being aware of the other agents\rq{} decisions. A game has
  \emph{complete information} if the number of agents, their payoffs and their feasible action spaces are common knowledge to all the agents.} As before, let
$\mathcal{F} = \{1,\dots,F\}$ denote the set of agents. Let the 
scalar-valued function $\theta^f(\chi^f, \chi^{-f})$ be the utility function of agent~$f \in \mathcal{F}$, which is a function
of all the agents actions $(\chi^f, \chi^{-f})$. The
feasible action space of agent $f \in \mathcal{F}$ is a set-valued
mapping $\Xi^f(\chi^{-f})$ with dimension~$n_f$
(in a regular Nash equilibrium problem, each agent's
feasible action space is a fixed set). Let $n := \sum_{f\in\mathcal{F}}
n_f$. Then $\Xi^f(\cdot)$ is a mapping from $\Real^{(n-n_f)}$ to
$\Real^{n_f}$. Parametrized by the other agents' decisions~$\chi^{-f}$,
each agent~$f\in\mathcal{F}$ in a GNEP solves the following problem:
\begin{equation} \label{gnep1}
\begin{aligned}
\maximize{\chi^f} \quad & \theta^f(\chi^f, \chi^{-f}) \\
\sbt \quad & \chi^f \in \Xi^f(\chi^{-f}).
\end{aligned}
\end{equation}
It is straightforward to see how the DSPI game in \eqref{dualpotential}
translates into a GNEP problem: for all $f \in \mathcal{F}$,
\begin{equation}
  \label{dspi-gnep}
  \begin{gathered}
    \chi^f = (x^f, y^f),\\
    \theta^f(\chi^f, \chi^{-f}) =  y^f_{t^f} - y^f_{s^f},\\[1.5ex]
    \Xi^f(\chi^{-f}) = \left\{ \chi^f = (x^f, y^f) \;\left|\;
      \begin{aligned}
        &  y^f_{v} - y^f_{u} \leq d_{uv}(x^f, x^{-f}) \quad \forall (u,v) \in A,\\
        & \sum_{(u,v) \in A} c^f_{uv} x^f_{uv} \leq b^f,\\
        & 0 \leq y^f_u \leq M \quad \forall u \in V,\\
        & x^f \in X^f
      \end{aligned}
    \right. \right\}.
  \end{gathered}
\end{equation}
Note that $\chi = (\chi^1,\dots,\chi^F) \in \Real^n$, where $n = F
\left(\vert V \vert + \vert A \vert \right)$. 


\sloppypar To formally define a Nash equilibrium to a GNEP, we let
$\Omega(\chi)$ denote the Cartesian product of the feasible sets of each
agent corresponding to decisions $\chi = (\chi^1, \ldots, \chi^F)$; that
is, 
\begin{equation}
\Omega(\chi) := \Xi^1(\chi^{-1}) \times \Xi^2(\chi^{-2}) \times \dots \times \Xi^F(\chi^{-F}).
\end{equation}
For a simultaenous-move GNEP with each agent solving
problem~\eqref{gnep1}, a generalized Nash equilibrium is defined as
follows: 

\begin{definition}
  A vector $\chi_N =  (\chi^1_N, \ldots, \chi^F_N) \in \Omega(\chi_N)$
  is a \emph{pure-strategy generalized Nash equilibrium (PGNE)} if for
  each agent~$f \in \mathcal{F} $,
  \begin{equation} \label{gne}
    \theta^f(\chi_N^f,\ \chi_N^{-f})  \geq   \theta^f(\chi^f,\ \chi_N^{-f}), \quad \forall \ \chi^f \in \Xi^f(\chi_N^{-f}).
  \end{equation}
\end{definition}
\hspace*{-8pt} Based on the above definitions and discussions, it is easy to see that if $(x, y)$ is an equilibrium to a DSPI game formulated as a GNEP using the primal-dual formulation, 
then $x$ is must be an equilibrium to the DSPI game using only the primal formulation. 
Such a relationship is formally stated below.
\begin{proposition} \label{GNEP_equivalence}
Suppose that $\chi = (x,y) \in \Real^{F \times (|A| + |V|)}$ is a PGNE to the GNEP where each agent solves \eqref{dspi-gnep}. Then $x$ is a PNE to the DSPI game where each agent solves \eqref{dni_paths}.\hfill$\Box$
\end{proposition}

For the remainder of the paper, we will mainly use the GNEP formulation, as our focus is shifting from establishing theoretical properties of DSPI equilibria to computing such an equilibrium. For continuous interdiction, the GNEP formulation can be further reformulated as an LCP, as we show below. On the other hand, under discrete interdiction, we show that certain classes of DSPI games admit provably convergent decentralized algorithms. In this case, we sequentially solve agents' problems \eqref{dspi-gnep} using an integer programming solver. 

\subsection{Linear Complementarity Formulation} \label{subsec:LCP}

Before presenting the LCP formulation for the DSPI game, we introduce
some basic notation and definitions. Formally, given a vector $q \in
\Real^d$ and a matrix $M \in \Real^{d \times d}$, a linear
complementarity problem LCP$(q,M)$ consists of finding a decision
variable vector $w \in \Real^d$ such that 
\begin{alignat}{2}
w & \geq  0, \label{lcpfeas1}\\
q + M w & \geq  0, \label{lcpfeas2}\\
w^T (q + M w) & =  0. \label{lcpcomp}
\end{alignat}
The LCP$(q,M)$ is said to be feasible if there exists a $w \in \Real^d$ that
satisfies \eqref{lcpfeas1} and \eqref{lcpfeas2}. Any $w$ satisfying
\eqref{lcpcomp} is called complementary. If $w$ is both feasible and
complementary, it is called a \emph{solution} of the LCP, and the set of such solutions 
is denoted by $ \text{SOL}(q,M)$. The LCP is said to be solvable if it has a solution. A
thorough exposition of the theory and algorithms for LCPs can be found
in \cite{cottle2009linear}.

Consider now the DSPI game with continuous interdiction, introduced in
Section~\ref{dspi}, where agent $f$'s optimization problem is given
in~\eqref{dualpotential3}. When the interdiction decisions of the agents $f' \neq
f$ are fixed, \eqref{dualpotential3} is a linear
program (LP).  In this case, the KKT conditions are both necessary and
sufficient for a given feasible solution to be optimal. 

In order to present the LCP reformulation in a more compact form, 
we introduce the following notation. Let $|V| = n$ and $|A| = m$. Denote by
$\mathcal{G}$ the arc-node incidence matrix of the graph $G$. 
Further let $\mathcal{I}$ denote an identity matrix, and $\mathbf{0}$ 
be a vector or a matrix of all zeros, of appropriate dimensions, respectively. 
The objective function coefficients for the LP~\eqref{dualpotential3},
denoted by $\phi^f \in \Real^{m+n}$ can be given as follows:
\begin{equation*}
\phi^f = \begin{bmatrix} \mathbf{0}_m \\ \nu^f \end{bmatrix}, \quad \text{where }
\quad \nu^f = \begin{cases} 1 &\mbox{if } u = s^f \\
        0 & \mbox{if } u \neq s^f, t^f \\
        -1 & \mbox{if } u = t^f
      \end{cases}.
\end{equation*}
The right hand sides for the constraints are denoted
as the vector $r^f(x^{-f}) \in \Real^{m+1}$:
{\renewcommand{\arraystretch}{1.2}
\begin{equation*} 
r^f(x^{-f}) = \begin{bmatrix} - d^0 \\ -b^f \end{bmatrix} -
\sum\limits_{\substack{f' \in \mathcal{F} \\ f' \neq f}} 
 \left[ \begin{array}{c|c}
 	\mathcal{I}_m & \mathbf{0}_{m \times n} \\ \hline
 	\mathbf{0}_m^T & \mathbf{0}_n^T
 \end{array} \right]	
 \begin{bmatrix} x^{f'} \\ y^{f'} \end{bmatrix}.
\end{equation*}}
The constraint matrix itself, denoted as $A^f \in \Real^{(m+1) \times (m+n)}$,
is 
{\renewcommand{\arraystretch}{1.3}
\begin{equation*} 
  A^f = \left[ \begin{array}{c|c}
    \mathcal{I}_m & \mathcal{G} \\ \hline
	-{c^f}^T & \mathbf{0}_{n}^T \end{array} 
  \right]. 
\end{equation*}}
Using this notation, the LP~\eqref{dualpotential3} can be restated as
follows:
\begin{equation} \label{dspi_compact}
\begin{array}{ll}
\minimize{x^f, y^f} & {\phi^f}^T \begin{bmatrix} x^f \\ y^f \end{bmatrix} \vspace{2 mm}\\
\text{subject to} & A^f \begin{bmatrix} x^f \\ y^f \end{bmatrix} \geq 
  r^f(x^{-f}), \vspace{3 mm}\\
  & \begin{bmatrix} x^f \\ y^f \end{bmatrix} \geq 0.
\end{array} 
\end{equation}
Let the dual variables for the LP \eqref{dspi_compact} be $(\lambda^f,
\beta^f, \upsilon^f)$, where $\lambda^f$ are the multipliers for the arc
potential constraints, $\beta^f$ the multiplier for the budgetary
constraint and $\upsilon^f$ the multipliers for the non-negativity
constraints. The KKT conditions for \eqref{dspi_compact} are given by
the following system.
\begin{equation} \label{KKT}
\begin{array}{rcl}
r^f(x^{-f}) \leq A^f  \begin{bmatrix} x^f \\ y^f \end{bmatrix}
	& \perp & \begin{bmatrix} \lambda^f \\ \beta^f \end{bmatrix} \geq 0, \vspace{3 mm}\\
0 \leq \begin{bmatrix} x^f \\ y^f \end{bmatrix} & \perp & \upsilon^f \geq 0, \\
\phi^f - {A^f}^T \begin{bmatrix} \lambda^f \\ \beta^f \end{bmatrix} - \upsilon^f & = & 0.
\end{array} 
\end{equation}
The KKT system~\eqref{KKT} can be rewritten in the following form:
\begin{equation} \label{KKT1}
\begin{array}{lll}
\upsilon^f = \phi^f - {A^f}^T \begin{bmatrix} \lambda^f \\ \beta^f \end{bmatrix} \geq 0, & \begin{bmatrix} x^f \\ y^f \end{bmatrix} \geq 0, & \begin{bmatrix} x^f \\ y^f \end{bmatrix}^T \upsilon^f = 0, \vspace{3 mm}\\
t^f = - r^f(x^{-f}) + A^f \begin{bmatrix} x^f \\ y^f \end{bmatrix} \geq 0, & \begin{bmatrix} \lambda^f \\ \beta^f \end{bmatrix} \geq 0, & {t^f}^T \begin{bmatrix} \lambda^f \\ \beta^f \end{bmatrix} = 0.
\end{array}
\end{equation}
In this form, it is easy to recognize that for a fixed value of
$x^{-f}$, the KKT system is equivalent to the LCP$(q^f(x^{-f}), M^f)$
where
\begin{equation} \label{lcpf}
  q^f(x^{-f}) = \begin{bmatrix} \phi^f \\ -r^f(x^{-f}) \end{bmatrix} \quad
    \text{ and } \quad M^f = \left[ \begin{array}{c|c}
    \mathbf{0}_{(m+n)\times(m+n)} & -{A^f}^T \\ \hline  A^f &
    \mathbf{0}_{(m+1) \times (m+1)} \end{array} \right].
\end{equation}
The decision variable vector for the LCP is the vector of combined
decision variables $ w^{f} := (x^f, y^f, \lambda^f, \beta^f)^T$. 
%
%
Each agent's KKT system~\eqref{KKT1} is parametrized by the collective
decisions of other agents. Now by stacking all agents' KKT systems
together, the resulting model is itself an LCP, which can be seen from
the following algebraic manipulation. 

First consider the
following system obtained from \eqref{KKT1} by expanding $r^f(x^{-f})$. 
{\renewcommand{\arraystretch}{1.2}
\begin{equation} \label{KKT2}
\begin{array}{lll}
\upsilon^f = \phi^f - {A^f}^T \begin{bmatrix} \lambda^f \\ \beta^f \end{bmatrix} \geq 0, & \begin{bmatrix} x^f \\ y^f \end{bmatrix} \geq 0, & \begin{bmatrix} x^f \\ y^f \end{bmatrix}^T \upsilon^f = 0, \vspace{3 mm}\\
t^f = \begin{bmatrix}  d^0 \\ b^f
\end{bmatrix} + A^f \begin{bmatrix} x^f \\ y^f \end{bmatrix} + 
\sum\limits_{\substack{f' \in \mathcal{F} \\ f' \neq f}} 
 \left[ \begin{array}{c|c}
 	\mathcal{I}_m & \mathbf{0}_{m \times n} \\ \hline
 	\mathbf{0}_m^T & \mathbf{0}_n^T
 \end{array} \right]	
 \begin{bmatrix} x^{f'} \\ y^{f'} \end{bmatrix} \geq 0, & \begin{bmatrix} \lambda^f \\ \beta^f \end{bmatrix} \geq 0, & {t^f}^T \begin{bmatrix} \lambda^f \\ \beta^f \end{bmatrix} = 0.
\end{array} \vspace*{-10pt}
\end{equation}}

\noindent Now introduce a matrix $\bar{M}^f$ to represent 
the interactions between agent $f$'s decision variables $(x^f, y^f)$ and the
KKT system of all the other agents, which has the following specific form: 
{\renewcommand{\arraystretch}{1.2}
\begin{equation} \label{mbar}
\bar{M}^{f} = \left[ \begin{array}{c|c|c|c}
		\mathbf{0}_{m \times m} & \mathbf{0}_{m \times n} &
    \mathbf{0}_{m \times m} & \mathbf{0}_{m \times 1} \\ \hline
		\mathbf{0}_{n \times m} & \mathbf{0}_{n \times n} &
    \mathbf{0}_{n \times m} & \mathbf{0}_{n \times 1} \\ \hline
    \mathcal{I}_m & \mathbf{0}_{m \times n} & \mathbf{0}_{m
                              \times m} & \mathbf{0}_{m \times 1} \\ \hline
    \mathbf{0}_{1 \times m} & \mathbf{0}_{1 \times n} & \mathbf{0}_{1 \times
        m} & 0 \end{array} \right].
\end{equation}}

\noindent Using this notation, the stacked KKT system \eqref{KKT2} for agents $f = 1,
\ldots, F$ can be formulated as an LCP$(q,M)$, with the vector $q$ and matrix $M$ given as follows:  
\begin{equation} \label{lcpq}
q = (\bar{q}^1, \bar{q}^2, \ldots, \bar{q}^F)^T, \quad \text{ where } \quad
\bar{q}^f = (\phi^f, d^0, b^f)^T, 
\end{equation}
and 
{\renewcommand{\arraystretch}{1.2}
\begin{equation} \label{lcpM}
M = \left[ \begin{array}{c|c|c|c|c}
      M^1 & \bar{M}^2 & \bar{M}^3 & \cdots & \bar{M}^F \\ \hline
      \bar{M}^1 & M^2 & \bar{M}^3 & \cdots & \bar{M}^F \\ \hline
      \vdots & \vdots & \vdots & \vdots & \vdots \\ \hline
      \bar{M}^1 & \bar{M}^2 & \cdots & \bar{M}^{F - 1} & M^F
      \end{array} \right].
\end{equation}}Due to the equivalence between an agent $f$'s optimization problem
\eqref{dualpotential3} and its KKT system \eqref{KKT1}, the above
LCP$(q,M)$ is equivalent to the corresponding (continuous) DSPI game in
the sense that a candidate point $(\chi^1, \chi^2, \ldots, \chi^F)$,
where $\chi^f = (x^f, y^f)$, is an equilibrium to the DSPI game if and
only if there exist associated Lagrangian multipliers such that they
together solve the LCP$(q,M)$. 

Methods for solving LCPs fall broadly into two categories: (i) pivotal
methods such as Lemke's algorithm, and (ii) iterative methods such as
splitting schemes and interior point methods. The former class of
methods are finite when applicable, while the latter class converge to
solutions in the limit. In general, the applicability of these
algorithms depends on the structural properties of the matrix $M$. In
the following analysis, we show that LCP$(q,M)$ for the DSPI game,
as defined in \eqref{lcpq} and \eqref{lcpM}, possesses two properties
that allow us to use Lemke's pivotal algorithm: (i) the matrix $M$ is a
copositive matrix, and (ii) $q \in (\text{SOL}(0,M))^*$.\footnote{Given a set $K
\in \Real^d$, the set $K^*$ denotes the \emph{dual} cone of K; i.e. $K^* = \{y \in
\Real^d: \ y^T x \geq 0, \ \forall x \in K \}.$}

We first show that $M$ is copositive. Recall that a matrix $M \in
\Real^{d \times d}$ is said to be \emph{copositive} if $x^T M x \geq 0$
for all $x \in \Real^d_+$.
\begin{lemma} \label{lcp_coplus}
$M$ defined as in \eqref{lcpM} is copositive.

\end{lemma}
\begin{proof}
Let $w \in \Real^{2m + n +1}_+$.
Using the block structure of $M$ given in \eqref{lcpM}, $w^T M w$ can be
decomposed as follows.
\begin{equation}
w^T M w = \sum_{f = 1}^F {w^f}^T M^f w^f + \sum_{f=1}^F
\sum\limits_{\substack{f' = 1 \\ f' \neq f}}^F {w^f}^T \bar{M}^{f'} w^{f'}.
\end{equation}
We analyze the terms in the two summations separately. First consider
${w^f}^T M^f w^f$ for any agent $f$. Let the dual variables $(\lambda^f,
\beta^f)$ be collectively denoted by $\delta^f$.
\begin{equation} \label{symmterm}
\begin{array}{rcl}
{w^f}^T M^f w^f & = & \left[ {\chi^f}^T \:\: {\delta^f}^T \right]
\left[ \begin{array}{c|c}
    \mathbf{0} & -{A^f}^T \\ \hline  A^f &
    \mathbf{0} \end{array} \right] 
\begin{bmatrix} \chi^f \\ \delta^f \end{bmatrix} \\
& = &  - {\chi^f}^T {A^f}^T \delta^f + {\delta^f}^T A^f \chi^f = 0.\\
\end{array}
\end{equation}
Now consider any term of the form ${w^f}^T \bar{M}^{f'} w^{f'}$: 
\begin{equation} \label{asymmterm}
\begin{array}{rcl} 
{w^f}^T \bar{M}^{f'} w^{f'} & = & 
  \left[ {x^f}^T \:\: {y^f}^T \:\: {\lambda^f}^T \:\: {\beta^f}^T \right] 
    \left[ \begin{array}{c|c|c|c}
		\mathbf{0} & \mathbf{0} & \mathbf{0} & 0\\ \hline 
        \mathbf{0} & \mathbf{0} & \mathbf{0} & 0\\ \hline
        \mathcal{I}_m & \mathbf{0} & \mathbf{0} & 0 \\ \hline
        \mathbf{0} & \mathbf{0} & \mathbf{0} & 0 \end{array} \right]
     \begin{bmatrix} x^{f'} \\ y^{f'} \\ \lambda^{f'} \\ \beta^{f'} \end{bmatrix} \\
     & = & \left[ {x^f}^T \:\: {y^f}^T \:\: {\lambda^f}^T \:\: \beta^f \right]
        \begin{bmatrix} \mathbf{0} \\ \mathbf{0} \\ x^{f'} \\ 0
        \end{bmatrix}  =  {\lambda^f}^T  x^{f'}.
\end{array}
\end{equation}
Combining \eqref{symmterm} and \eqref{asymmterm} we obtain
\begin{equation} \label{coplussum}
 w^T M w = \sum_{f=1}^F \sum\limits_{\substack{f' = 1 \\ f' \neq f}}^F 
  {\lambda^f}^T  x^{f'}.
\end{equation}
Since $\lambda^f$'s and $x^{f'}$'s are the elements of $w$, 
$w \geq 0$ clearly implies that $w^T M w \geq 0$.  
\end{proof}

We now show property (ii) of the LCP$(q, M)$; that is, $q \in (\text{SOL}(0,M))^*$. 
\begin{lemma} \label{qsol}
Let the vector $q$ and the matrix $M$ be as defined in \eqref{lcpq} and
\eqref{lcpM} respectively. Then $q \in (\text{SOL}(0,M))^*$.
\end{lemma}
\begin{proof} 
First note that $\text{SOL}(0,M)\neq \emptyset$ for any $M$, since $0$ is always a solution to LCP$(0,M)$.  Now 
consider a $w \in \text{SOL}(0,M)$; i.e. $0 \leq w \perp 0 + M w \geq 0$. 
We prove that $q^T w \geq 0$. Observe that $q^T w$ can be decomposed as follows:
\begin{equation}\label{qw}
\begin{array}{rcl}
  q^T w & = & \displaystyle \sum_{f=1}^F \bar{q}^{f^T} w^f  =  \displaystyle \sum_{f=1}^f \left( \phi^{f^T} \begin{bmatrix} x^f \\ y^f \end{bmatrix} +
    d^{0^T} \lambda^f + b^f \beta^f \right)  \\
    & = & \displaystyle \sum_{f=1}^F\left[ (y^f_{s^f} - y^f_{t^f}) + d^{0^T} \lambda^f
    + b^f \beta^f\right].
\end{array}
\end{equation}
The last two terms in the last equality above, $d^{0^T} \lambda^f$ and
$b^f \beta^f$, can be easily seen to be nonnegative for $f = 1, \ldots
F$. This is so because $w \in \text{SOL}(0,M)$ implies that
$\lambda^f, \beta^f \geq 0$, and by assumption $d^0, b^f \geq 0$  for
each $f = 1 ,\dots, F$. 

Now we focus on the first term in the last equality of \eqref{qw}: 
$\sum_{f=1}^F (y^f_{s^f} - y^f_{t^f})$. First since $M w \geq 0$, $w^f$
must solve the system~\eqref{KKT2} for $f = 1, \ldots, F$, with
$\phi^f$, $d^0$ and $b^f$ all set at zeros (as $\phi^f$, $d^0$ and $b^f$
are the components of the vector $q$ in the LCP$(q, M)$, as defined in
\eqref{lcpq}; and in LCP$(0, M)$,  $q=0$). In this case, considering the
primal feasibility of $w^f$, we obtain the following:
\begin{equation} \label{KKT0} \left.
\begin{array}{rcl}
  \displaystyle \sum_{a \in A} c^f_a x^f_a & \leq & 0 \\
  \displaystyle y^f_u - y^f_v + \sum_{f=1}^F x^f_{u,v} & \geq & 0 \quad \forall

  (u,v) \in A
\end{array} \right\} \quad \text{for } f = 1, \ldots, F.
\end{equation}
Recall that $c^f_a \geq 0$ for all $a \in A$ and $f = 1, \ldots, F$ by
assumption. Therefore, \eqref{KKT0} implies that $x^f = 0$ for any agent~$f$.
It is easy to see that in this case, we must have 
\begin{equation} \label{nonnegpot}
y^f_u - y^f_v \geq 0 \quad \forall (u,v) \in A, \text{ for } f = 1, \ldots F.
\end{equation}
Now consider any $s^f$-$t^f$ path $\mathcal{P}^f$. By assumption, there
must be at least one such path for each agent~$f$. By summing up the
inequalities~\eqref{nonnegpot} over the arcs in the path
$\mathcal{P}^f$, we obtain the desired result. In other words,
\begin{equation}
 \displaystyle \sum_{(u,v) \in \mathcal{P}^f} y^f_u - y^f_v = y^f_{s^f} -
 y^f_{t^f} \geq 0.
\end{equation}
Summing up over the agents, we thus have shown that $q^T w \geq 0$ for any $w \in \text{SOL(0, M)}$, which completes the proof.
\end{proof}

With Lemma \ref{lcp_coplus} and \ref{qsol}, we can apply the following result
from Cottle et al. \cite{cottle2009linear}.
\begin{theorem} (\cite{cottle2009linear}, Theorem 4.4.13) If $M$ is copositive and $q \in (\text{SOL}(q,M))^*$, then Lemke's method will
  always compute a solution, if the problem is nondegenerate.\footnote{A detailed
  discussion of degeneracy and cycling in Lemke's method can be found
in Section 4.9 of \cite{cottle2009linear}.}
\end{theorem}

As discussed earlier, the LCP approach is not applicable for discrete
DSPI games due to the lack of necessary and sufficient optimality
conditions. In the following we develop a decentralized approach that
works for both discrete and continuous DSPI games.

\subsection{Gauss-Seidel Algorithm (Algorithm 1)}

We first present the basic form of a best response based algorithm.
The idea is simple: starting with a particular feasible decision vector $\chi_0 = (\chi_0^1, \chi_0^2, \ldots,
\chi_0^F) \in \Omega(\chi_0)$, solve the optimization problem of a particular agent, say,
agent 1, with all of the other agents' actions fixed. Assume an optimal
solution exists to this optimization problem, and denote it as
$\chi^{1*}_1$. The next agent, say, agent 2, solves its own optimization
problem, with the other agents' actions fixed as well, but with
$\chi_0^1$ replaced by $\chi^{1*}_1$. Such an approach is often referred
to as a diagonalization scheme or a Gauss-Seidel iteration, and for
the remainder of this paper we use the latter name to refer to this
simple best-response approach. 

Consider applying the Gauss-Seidel iteration to a GNEP, with each agent 
solving the optimization problem \eqref{gnep1}, denoted as $\mathcal{P}(\chi^{-f})$. The Gauss-Seidel
iterative procedure is presented in Algorithm 1 below.

\begin{algorithm}
  \caption{Gauss-Seidel Algorithm for a GNEP} \label{GS_convex}
  \begin{algorithmic}[0]
    \State Initialize. Choose $\chi_0 = (\chi_0^1, \ldots, \chi_0^F)$ with $\chi_0^f \in \Xi^f(\chi_0^{-f})$\ $\forall f \in \mathcal{F}$. Set $k \leftarrow 0$.

    \smallskip
    \State Step 1:
    \For{$f=1,2,\ldots,F$} 
    \State Set $\chi_{k,f}^{-f} \leftarrow (\chi^1_{k+1}, \ldots, \chi^{f-1}_{k+1}, \chi^{f+1}_k, \ldots, \chi^F_k)$;
    \State Solve $\mathcal{P}(\chi_{k,f}^{-f})$ to obtain an optimal solution $\chi_{k,f}^*$;
    \If{$\theta^f(\chi_{k,f}^*, \chi^{-f}_{k,f}) \ > \ \theta^f(\chi^f_{k}, \chi^{-f}_{k,f})$} Set $\chi_{k+1}^f \leftarrow \chi_{k,f}^*$; 
 \Else \ Set $\chi_{k+1}^f \leftarrow \chi_k^f$;
    \EndIf
    \EndFor
    \State Set $ \chi_{k+1} \leftarrow (\chi^1_{k+1}, \ldots, \chi^F_{k+1}). $
    \State Set $k \leftarrow k+1$.

    \smallskip
    \State {\bf if} $\chi_k$ satisfies termination criteria, {\bf then STOP}; {\bf else} \textbf{GOTO} Step 1.
  \end{algorithmic}
\end{algorithm}

Note that updates in agent $f$'s decisions occur at iteration $k$ only if there
is a strict increase in the agent's payoff at the iteration.
The algorithm can be directly applied to compute an
equilibrium of a DSPI game with discrete interdiction. For finite
termination, we fix a tolerance parameter $\epsilon$ and use the
following stopping criterion: 
\begin{equation}\label{termin}
  \norm{\chi_k - \chi_{k-1}} \leq \epsilon. 
\end{equation}

\begin{proposition} \label{GS_heuristic}
  Suppose that the Gauss-Seidel algorithm (Algorithm~\ref{GS_convex}) 
  is applied to the DSPI game with discrete
  interdiction, and the termination criterion \eqref{termin} is used
  with $\epsilon < 1$. If the algorithm terminates at $\chi_k$, then
  $\chi_k$ is an equilibrium to this problem.
\end{proposition} 

\begin{proof}
  Since the variables $\chi_k$ are integral for discrete interdiction
  problems, choosing $\epsilon < 1$ for the termination criterion will
  ensure that the algorithm terminates only when successive outer
  iterates are equal. Consequently, by the assumption, $\chi_{k-1} =
  \chi_k$ at termination. This also implies that $\chi_{k-1,f}^{-f} =
  \chi_k^{-f}$ for $f = 1, \ldots, F$. By construction of $\chi_k$, we
  must then have
  \begin{equation*}
    \chi_k^f = \amin_{\chi^f \in \Xi^f(\chi_k^{-f})}
    \theta^f(\chi^f, \chi_k^{-f}). 
  \end{equation*}
  Clearly, $\chi_k$ must then be an equilibrium.
\end{proof}

Even though there is no guarantee that the algorithm will in fact converge, 
we note that in the discrete case, it is possible to detect when the
algorithm fails to converge. Recall that $\Xi^f(\chi^{-f}) \subseteq
K^f$ for each agent $f \in \mathcal{F}$, where $K^f$ is defined below.
  \begin{equation} 
    \label{kf}
    K^f = 
    \left\{ (x^f, y^f) \  \Bigg\vert \ 
      \begin{aligned}
        \sum_{(u,v) \in V} c^f_{uv} x^f_{uv} & \leq b^f, \\
        0 \leq y^f_u & \leq M \quad \forall u \in V
      \end{aligned} 
    \right\}. 
  \end{equation}
  Clearly, the set $\prod_{f=1}^F K^f$ is finite. Any
intermediate point $\chi_k$ generated by Algorithm~\ref{GS_convex} must
certainly satisfy the budgetary constraints on $x^f_k$ and the bound
constraints on $y^f_k$ for each agent $f$. Therefore $\chi_k \in
\prod_{f=1}^F K^f$. In other words, the set of possible points
$\chi_k$ generated by Algorithm~\ref{GS_convex} lies in a finite set. This
means that if the algorithm fails to converge, it must generate a
sequence that contains at least one cycle. The existence of such
cycles in non-convergent iterate paths can then be used to detect
situations in which the algorithm fails to converge.

Proposition \ref{GS_heuristic} is likely the best one can do
for general DSPI games under discrete interdiction. However, for the subclass of
  such games with common source-target pairs, we can in fact prove that the best
  response dynamics always terminates in a NE in a finite number of steps.

\begin{proposition} \label{GS_commonst}
  Consider a DSPI game with discrete interdiction with common source-target
  pairs, and assume that the initial arc lengths $d$ and arc extensions $e$
  are integral. Suppose that Algorithm~\ref{GS_convex} is applied to such a problem, and the termination criteria \eqref{termin} is used with $\epsilon < 1$.
  Then the algorithm will terminate finitely at an equilibrium.
\end{proposition}
\begin{proof}
Denote the common source node as $s$, and the common target
node as $t$. The set of joint
feasible strategies in $x$ under the given assumptions is a finite set.
Moreover, all the agents attempt to minimize the common objective, namely the
$s$-$t$ path length. Note that at any iteration $k$ at which an update occurs for
any agent's decision, there must then be a strict increase in the $s$-$t$ path
length. Thus there can be no cycles in the sequence $\{\chi_k\}$. Furthermore,
since the set of joint feasible strategies is finite, the sequence must
terminate at some point $\chi^*$. It is easy to show that $\chi^*$ must be an
equilibrium (cf.~Proposition \ref{GS_heuristic}).
\end{proof}

\subsection{Regularized Gauss-Seidel Algorithm (Algorithm 2)}

One disadvantage of the
``na\"{i}ve'' Gauss-Seidel algorithm described above is that for continuous GNEPs,
it can fail to converge to equilibria. However, Facchinei et al.~\cite{facchinei_gpg}
showed that under certain assumptions, we can overcome this issue by
adding a regularization term to the individual agent's problem solved in
a Gauss-Seidel iteration.

The regularized version of the optimization problem for agent $f \in
\mathcal{F}$ is 
\begin{equation} \label{reg_sub}
  \begin{aligned}
    \maximize{\chi^f} \quad & \theta^f(\chi^f, \chi^{-f}) - \tau \norm{\chi^f - \overline{\chi}^f}^2 \\
    \text{subject to} \quad & \chi^f \in \Xi^f(\chi^{-f}),
  \end{aligned}
\end{equation} 
where $\tau$ is a positive constant. Here the regularization term is
evaluated  in relation to a candidate point $\overline{\chi}^f$.
Note that the point $\overline{\chi}^f$ and the other agents' decision
variables $\chi^{-f}$ are fixed when the problem \eqref{reg_sub} is
solved.  We refer to problem~\eqref{reg_sub} as $\mathcal{R}(\chi^{-f},
\overline{\chi}^f)$. 
The regularized Gauss-Seidel procedure, herein referred to as Algorithm 2, 
is simply Algorithm 1, except that $\mathcal{R}(\chi^{-f},
\overline{\chi}^f)$ is solved in each iteration $k$ instead of $\mathcal{P}(\chi_{k,f}^{-f})$, 
with $\overline{\chi}^f$ given by $\chi_k^f$, for each $f$. 

This version of the algorithm, along with its convergence proof, was originally presented in
\cite{facchinei_gpg} to solve GNEPs with shared constraints. 
The difficulty here that prevents us from showing convergence lies in the
fact that we are dealing with GNEPs with non-shared constraints.  
As a result, any intermediate points resulting from an agent's best responses need not to be feasible in the other agents' problems.
Consequently, we use Algorithm 2 only as a heuristic algorithm to solve
DSPI games under continuous interdiction. 
Nevertheless, we can show that if
Algorithm 2 converges, then the resulting point is an equilibrium
to the DSPI game.
\begin{proposition}\label{Exist_Cont}
   Let $\{ \chi_k \}$ be the sequence generated by applying
  Algorithm 2 to the DSPI problem under continuous interdiction,
  wherein each agent solves \eqref{dualpotential}. Suppose $\{ \chi_k \}$ converges to
  $\bar{\chi}$. Then $\bar{\chi}$ is an equilibrium to the DSPI problem.
\end{proposition}

The proof of this proposition is almost identical to that of
Theorem 4.3 in Facchinei et al.~\cite{facchinei_gpg}. However, we do want to point out one key difference in the
proof. In Proposition \ref{Exist_Cont},  we need to assume that the
entire sequence $\{ \chi_k \}$ converges to $\bar{\chi}$. This is a
strong assumption in the sense that it also requires that all the
intermediate points $\chi_{k,f}$ in Algorithm 2  to
converge to $\bar{\chi}$, a fact key to proving that $\bar{\chi}$ is
indeed an equilibrium. In contrast, for GNEPs with shared constraints,
this assumption may be weakened because the intermediate points
$\chi_{k,f}$ and therefore the cluster points of the sequence generated
by the algorithm are guaranteed to be feasible. 
The complete proof is presented in Appendix
\ref{proof_exist_cont}.

Similar to discrete DSPI games, the convergence of
Algorithm 2  can be guaranteed for continuous-interdiction DSPI games with common
source-target pairs. The key fact that allows us to
prove this stronger result is that by dropping the dependence of 
the variables $y$ on the agents $f \in F$, any unilateral deviation in the shared variables
$y$ results in a solution that remains feasible in the other agents' optimization problems. 
The convergence result is formally stated below.

\begin{proposition}\label{Exist_Cont_Common}
 Consider applying Algorithm 2 to the DSPI problem under
 continuous interdiction with common source-target pairs, where each agent solves
 \eqref{dni_paths}. Let $\{ \chi_k \}$ be the sequence generated by
 the algorithm. If $\bar{\chi}$ is a cluster point of this sequence, then it
 also solves the DSPI problem. \hfill$\Box$
\end{proposition}


\subsection{Numerical Results}
\label{sec:numerical}

We use the algorithms presented in the previous section to study several
instances of DSPI games. The decentralized algorithms were implemented
in MATLAB R2010a with CPLEX v12.2 as the optimization solver. The LCP
formulation for the DSPI game with continuous interdiction was solved
using the MATLAB interface for the complementarity solver PATH
\cite{ferris1999interfaces}. Computational experiments were carried out
on a desktop workstation with a quad-core Intel Core i7 processor and 16
GHz of memory running Windows 7.

In the implementation of the decentralized algorithm, for DSPI games
with discrete interdiction, we used Algorithm~\ref{GS_convex} with $\tau
= 0$. For DSPI games with continuous interdiction, we followed a
strategy of trying the ``na\"{i}ve'' Gauss-Seidel algorithm -- i.e.
Algorithm~\ref{GS_convex} with $\tau = 0$ -- first. If this version
failed to converge in 1000 outer iterations, we then set $\tau$ to a
strictly positive value and used the more expensive regularized
Gauss-Seidel algorithm. 

\subsubsection*{Computing Equilibria}

First, we applied the algorithm to Example~\ref{example1} in Section
\ref{uniqueness}, which is a DSPI game with continuous interdiction. 
In particular, the network is given in
Figure~\ref{fig1} and there are 2 agents: agent~1 has an adversary with
source node 1 and target node 5, and agent~2 has an adversary with
source node 1 and target node 6. Both agents have an interdiction budget
of 1. The initial arc lengths are 0, and the interdiction costs are
equal for both agents and are given as the arc labels in Figure
\ref{fig1}, with $\epsilon = 2$. We set the regularization parameter
$\tau = 0.01$. We were able to obtain a solution within an accuracy of
$10^{-6}$ in 3 outer iterations.

Furthermore, we obtained multiple Nash equilibria by varying the
starting point of the algorithm. All the equilibria obtained resulted in
the same shortest path lengths for each agent. Some of the equilibria
obtained are given in Table \ref{multi_eq_table}. The column $x_0$
represents the starting interdiction vector for each agent, 
the columns $x_N^1$ and $x_N^2$ give the equilibrium interdiction
vectors for agents 1 and 2, respectively. The seven components in the
vectors of $x_0$, $x_N^1$ and $x_N^2$ represent the interdiction actions
at each of the seven arcs in Figure \ref{fig1}, with the arcs being
ordered as follows: first, the top horizontal arcs $(1,2)$ and $(2,3)$,
then the vertical arcs $(1,4)$, $(2,5)$ and $(3,6)$, and finally the
bottom horizontal arcs $(4,5)$ and $(5,6)$. The remaining two columns in
Table \ref{multi_eq_table} , $p_1$ and $p_2$, give the shortest path
lengths for agents 1 and 2 respectively, at the equilibrium $\chi_N$.

\begin{table}[!h]
\caption{Multiple equilibria for the instance of the DSPI game in Example 2} \label{multi_eq_table}
\footnotesize
\centering
\begin{tabular}{ccccc}
\hline
$x_0$  & $x_N^1$ & $x_N^2$ & $p_1$ & $p_2$ \\
\hline
$(0, 0, 0, 0, 0, 0, 0)$ & $(0, 0, 0.5, 0.5, 0, 0, 0)$ & $(0, 0, 0.1667, 0.1667, 0.6667, 0)$ & $0.6667$ & $0.6667$ \\
$(0.2, 0.2, 0, 0, 0, 0, 0)$ & $(0, 0, 0.6, 0.4, 0, 0, 0)$ & $(0, 0, 0.0667, 0.2667, 0.6667, 0)$ & $0.6667$ & $0.6667$ \\
$( 0, 0, 0, 0, 0, 0.2, 0.2)$ & $(0, 0, 0.4, 0.6, 0, 0, 0)$ & $(0, 0, 0.2667, 0.0667, 0.6667, 0)$ & $0.6667$ & $0.6667$ \\
$(0, 0, 0, 0, 0, 0.3, 0.3)$ & $(0, 0, 0.35, 0.65, 0, 0, 0)$ & $(0, 0, 0.3167, 0.0167, 0.6667, 0)$ & 0.6667 & 0.6667 \\
$(0.3, 0.3, 0, 0, 0, 0, 0)$ & $(0, 0, 0.65, 0.35, 0, 0, 0)$ & $(0, 0, 0.0167, 0.3167, 0.6667, 0)$ & 0.6667 & 0.6667 \\
$(0.25, 0.25, 0, 0, 0, 0, 0)$ & $(0, 0, 0.625, 0.375, 0, 0, 0)$ & $(0, 0, 0.0417, 0.2917, 0.6667, 0)$ & 0.6667 & 0.6667 \\
$(0, 0, 0, 0, 0, 0.25, 0.25)$ & $(0, 0, 0.375, 0.625, 0, 0, 0)$ & $(0, 0, 0.2917, 0.0417, 0.6667, 0)$ & 0.6667 & 0.6667 \\
$(0, 0, 0, 0, 0, 0.15, 0.15)$ & $(0, 0, 0.425, 0.575, 0, 0, 0)$ & $(0, 0, 0.2417, 0.0917, 0.6667, 0)$ & 0.6667 & 0.6667 \\
$(0.15, 0.15, 0, 0, 0, 0, 0)$ & $(0, 0, 0.575, 0.425, 0, 0, 0)$ & $(0, 0, 0.0917, 0.2417, 0.6667, 0)$ & 0.6667 & 0.6667 \\
\hline
\end{tabular}
\end{table}

\begin{figure}[!h]
\centering
\includegraphics[width=0.9\textwidth]{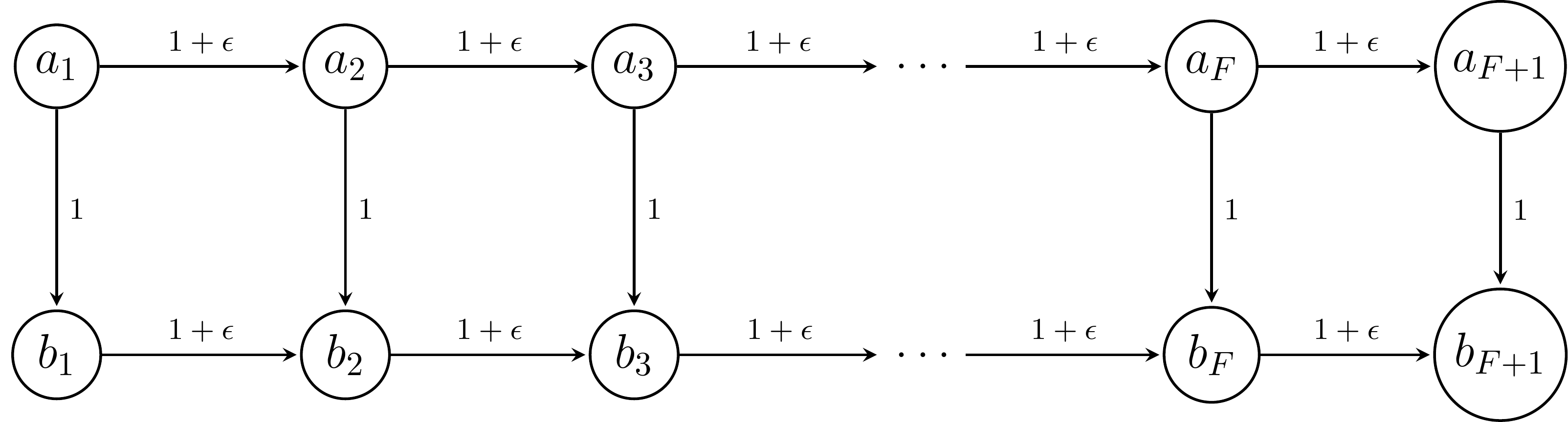}
\caption{Network structure for DSPI Example \ref{example4}} \label{dspi_graph1_fig}
\end{figure}

\vspace{0.2 in}
\begin{example} 
\label{example4} \rm{To test the algorithm on larger-scale problems, we
expanded the instance in Example \ref{example1} to larger network sizes
and numbers of agents. For $F$ agents, the graph contains $2(F+1)$
vertices with the edges as shown in Figure~\ref{dspi_graph1_fig}. The
source vertex for all agents is $a_1$. The target vertex for a given
agent $f$ is $b_{f+1}$. The initial arc lengths are all assumed to be
zero. The interdiction costs are the same for all the agents and are
given as the arc labels in Figure~\ref{dspi_graph1_fig}. All the agents
have an interdiction budget of 1. The cost parameter $\epsilon$ is
chosen as 2. For discrete interdiction on these graphs, the arc extensions
are assumed to be length 1.

The running time and iterations required to compute equilibria for these
instances are summarized in Table~\ref{dspi_ex2}. The first four columns
in the table give the number of outer iterations and runtime for
Algorithm~\ref{GS_convex} over these instances with continuous
interdiction. The results indicate that the running time for the
centralized Lemke's method increases monotonically with the problem
size. However, the running time for the decentralized method depends not
just on the problem size but also on the number of outer iterations. In
general, there is no correlation between these two parameters. Indeed
the algorithm is observed to converge in relatively few iterations even
for some large problem instances. This is in stark contrast to the rapid
increase in running time observed for the LCP approach as problem size
increases.

It must be noted that the order in which the individual agent problems
are solved in the Gauss-Seidel algorithm plays an important role. Indeed
it was found that the algorithm failed to converge for certain orderings
of the agents, but succeeded in finding equilibria quickly for the same
instance with other orderings. For instance, for a network of size
25, solving the agent problems in their natural order $\{1, 2, \ldots,
25 \}$ resulted in the failure of the ``na\"{i}ve'' version of the
algorithm to converge even after 1000 outer iterations. However, with a
randomized agent order, the algorithm converged in as few as 13
iterations. It is encouraging to note that for the same agent order that
resulted in the failure of the naive version, the regularized method
converged to a GNE within 394 outer-iterations with a runtime of 28
wall-clock seconds.}

\begin{table}[!h]
\caption{Number of iterations and running times for DSPI Example 3.} \label{dspi_ex2}
\centering
\begin{tabular}{c|cc:c|cc}
\hline
& \multicolumn{3}{c|}{Continuous Interdiction}  & \multicolumn{2}{c}{Discrete
Interdiction} \\
\hline
& \multicolumn{2}{c:}{Decentralized} & LCP & \multicolumn{2}{c}{Decentralized} \\   
\# Agents& \# Iters & Runtime (s) & Runtime (s)  & \# Iters & Runtime (s)\\
\hline
5 & 3 & 0.0205 & 0.0290 & 5     & 0.1776 \\ 
10 & 5 & 0.0290 & 0.1833 & 3     & 0.1627 \\ 
15 & 11 & 0.1103 & 0.7534 & 3     & 0.2419 \\ 
20 & 5 & 0.0723 & 2.1106 & 3     & 0.3164 \\ 
25 & 13 & 0.2609 & 4.8167 & 3     & 0.4005 \\ 
30 & 15 & 0.4070 & 10.2256 & 3     & 0.5155 \\ 
35 & 10 & 0.3605 & 17.7387 & 3     & 0.5948 \\ 
40 & 41 & 1.7485 & 30.2382 & 3     & 0.7387 \\ 
45 & 12 & 0.6601 & 48.6280 & 3     & 0.8794 \\ 
50 & 12 & 0.7981 & 75.0420 & 3     & 1.0385 \\
\hline
\end{tabular}
\end{table}

\end{example}

\subsubsection*{Computation of Efficiency Losses}

Using the decentralized algorithm and its potential to find multiple
equilibria by starting at different points, we empirically study the
efficiency loss of decentralized interdiction strategies in DSPI games.
We focus first on Example~\ref{example4}, with the underlying network
represented in Figure~\ref{dspi_graph1_fig}.  Before computing the
average efficiency losses, we first
establish a theoretical bound on the worst-case price of
anarchy, for the purpose of comparison.

We start with the specific instance as depicted in Figure \ref{dspi_graph1_fig}. 
Recall that there are $F$ agents and the source-target pair for agent
$f$ is $(a_1, b_{f+1})$. Note that all paths for all agents contain
either the arc $(a_1, a_2)$ or the arc $(a_1, b_1)$. Then one feasible
solution to the centralized problem is for each agent to interdict both
these arcs by $1 / (2+\epsilon)$ for a total cost of 1. In this case the
length of both arcs become $F / (2+\epsilon)$, giving a shortest path
length of $F / (2 + \epsilon)$ for each agent. Note that this is not an
equilibrium solution as agent 1 can deviate unilaterally to interdict
arcs $(a_1, b_1)$ and $(a_2, b_2)$ by $1/2$ to obtain a shortest path
length of $(F+\epsilon / 2)/(2 + \epsilon)$. 

A Nash equilibrium to this problem is given by the following solution.
Agent $f$ interdicts the vertical arcs $(a_1, b_1), \ldots, (a_f, b_f)$ by
$1 / (f(f+1))$ and the arc $(a_{f+1}, b_{f+1})$ by $f / (f+1)$.
Each agent then has a shortest path length of $F / (F+1)$. Note that
all the $s^f$-$t^f$ paths are of equal length for every agent. Therefore
diverting any of the budget to any vertical arcs will result in unequal
path lengths and a shorter shortest path for any agent. Obviously,
diverting the budget to interdict any of the horizontal arcs is cost
inefficient because of their higher interdiction cost at $1 + \epsilon$.
Thus no agent has an incentive to deviate from this solution.

We now have a feasible solution to the centralized problem that has an
objective value of $F / (2 + \epsilon)$ for each agent, and a Nash
equilibrium that has an objective value of $F / (F+1)$ for each agent.
Therefore, by its definition in \eqref{wpoa}, the worst-case price of anarchy for the DSPI game depicted in
Figure \ref{dspi_graph1_fig} must be at least $(F+1) / (2 + \epsilon)$.

Using the regularized Gauss-Seidel algorithm we also compute lower
bounds on the worst-case price of anarchy and average efficiency losses for
the same network topology with varying number of agents. 
The instances we consider are obtained by varying $\epsilon$
uniformly in the range of $(1.5, 10)$. For the purpose of comparison,
the numerical results are plotted in Figure \ref{ex12} below. Note that
the average-case efficiency loss is much lower than the worst-case price
of anarchy. For the particular graph structure under consideration, we
observe that the average efficiency loss grows at a much lower rate than
the worst-case efficiency loss. However this observation cannot be
generalized to other graph structures and such patterns may only be
discernible by applying a decentralized computational framework as the
one we presented. 

\begin{figure}[!h] 
\includegraphics[width=\textwidth]{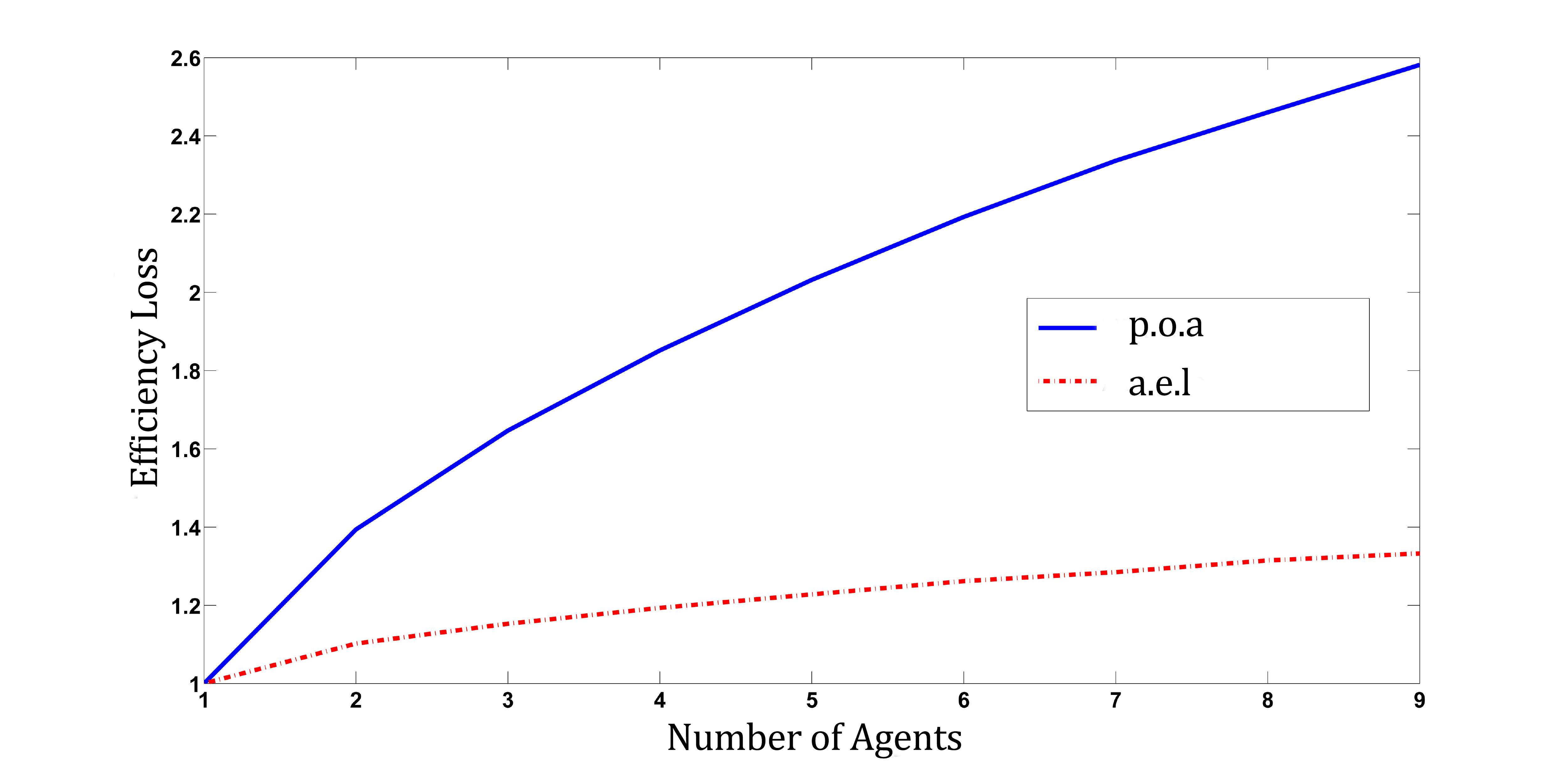}
\caption{Efficiency loss with respect to the number of agents.}
 \label{ex12}
\end{figure}

\begin{example} \label{example5} 
\rm{We further tested the decentralized algorithms for continuous
interdiction on random graphs to study average efficiency losses of
equilibria of DSPI games on networks with different topologies. For the
random graphs we generated, the input parameters include the number of
vertices and the density of a graph, which is the number of arcs divided
by the maximum possible number of arcs. The number of agents was chosen
randomly from the interval $(0, |V|/2)$, and one such number is chosen
per vertex set size. Source-target pairs were chosen at random for each
interdictor. Fixing the vertex set, we populated the arc set by
successively generating source-target paths for the agents until the
desired density was reached. We thus ensured connectivity between the
source-target pairs for each agent. Costs, initial arc lengths and
interdiction budgets were chosen from continuous uniform distributions.
Arc interdiction costs were assigned uniformly in the range $[1, 5]$.
The budget for each agent $f$ was chosen uniformly from the interval
$[b^f/10, b^f/2]$, where $b^f = \sum_{a \in A} c_a^f$. The initial
length of each arc was chosen uniformly from $[1,5]$. 
	
For each combination of vertex set size, the number of
agents, and graph density, we generated 25 random instances by drawing
values from the uniform distributions described above for the various
network parameters. For each instance, we used 10 different random permutations
of the agents to run the decentralized algorithms in an attempt to
compute multiple equilibria. The lower bound on the price of anarchy for
the game was computed as the worst case efficiency loss over these 25
instances. The average efficiency loss over these instances was also
computed. The results are summarized in Table \ref{dspi_ex22}. Our
experiments indicate that the average efficiency loss and the worst-case
price of anarchy tend to grow as the number of vertices and number of
agents increases; on the other hand, these measures of efficiency loss
sometimes do not appear to be monotonically increasing or decreasing
with respect to the density of the underlying network. 

\begin{table}[h]
\footnotesize
\caption{DSPI Continuous Interdiction - Random Graphs} \label{dspi_ex22}
\begin{tabular*}{\textwidth}{@{\extracolsep{\fill}}ccccccc}
\hline
\# Vertices  & \# Agents & Density & Avg. Run Time (s) & \# Avg Iters. & $a.e.l$ & $p.o.a$  \\
\hline
    5     & 3     & 0.25  & 0.0037 & 3 & 1.3133 & 1.5561 \\
    5     & 3     & 0.5   & 0.0038 & 3  & 1.3265 & 1.9529 \\
    5     & 3     & 0.75  & 0.0040 & 3  & 1.5099 & 2.3829 \\
    10    & 3     & 0.25  & 0.0065 & 4 & 1.5366 & 2.2078 \\
    10    & 3     & 0.5   & 0.0176 & 11 & 1.4538 & 2.3114 \\
    10    & 3     & 0.75  & 0.0132 & 8 & 1.4273 & 2.1342 \\
    15    & 4     & 0.25  & 0.0263 & 11 & 1.7091 & 2.9246 \\
    15    & 4     & 0.5   & 0.0939 & 33 & 1.7524 & 2.7904 \\
    15    & 4     & 0.75  & 0.1267 & 42 & 1.5695 & 2.1425 \\
    20    & 5     & 0.25  & 0.1269 & 34 & 2.1907 & 3.2885 \\
    20    & 5     & 0.5   & 0.2087 & 43 & 1.8523 & 2.7906 \\
    20    & 5     & 0.75  & 0.5416 & 100 & 1.7967 & 2.3782 \\
    25    & 7     & 0.25  & 0.7167 & 105 & 2.5631 & 4.8788 \\
    25    & 7     & 0.5   & 1.9564 & 207 & 2.3022 & 5.5794 \\
    25    & 7     & 0.75  & 1.8476 & 158 & 1.9884 & 2.4423 \\

\hline
\end{tabular*}
\end{table}
}
\end{example}

\section{Conclusions and Future Work}

In this work, we introduced decentralized network interdiction (DNI)
games and gave formulations for one such class of games -- decentralized
shortest path interdiction (DSPI) games.  We analyzed the theoretical
properties of DSPI games: in particular, we gave conditions for the
existence of equilibria and examples where multiple equilibria exist.
Specifically, we proved the existence of equilibria for general DSPI games under
continuous interdiction. On the other hand, for the discrete counterpart, we
provide a counterexample for existence. However, for the subclass of problems
with common source-target pairs, we are able to provide an existence guarantee.  

We also showed that the DSPI game under continuous interdiction is equivalent to
a linear complementarity problem, which can be solved by Lemke's algorithm. This
constitutes a convergent centralized method to solve such problems. We also
presented decentralized heuristic algorithms to solve DSPI games under both
continuous and discrete interdiction.  Finally, we used these algorithms to
empirically evaluate the worst case and average efficiency loss of DSPI games.

There are other classes of network interdiction games that can be
studied using the same framework we have developed, where the 
agents' obstruction functions are related to the maximum flow or minimum
cost flow in the network. 
Establishing theoretical results and studying the applicability of the
decentralized algorithms to other classes of decentralized network
interdiction games are natural and interesting extensions of this work.

In our study of DSPI games, we also made the assumption
that the games have complete information structure; that is, the normal form of the
game -- the set of agents, agents' feasible action spaces, and their
objective functions -- is assumed to be common knowledge to all agents.
In addition, we made the implicit assumption that all input data are
deterministic. However, data uncertainty and lack of observability of
other agents' preferences or actions are prevalent in real-world
situations. For such settings, we need to extend our work to accommodate
games with exogenous uncertainties and incomplete information.

One might also be interested in designing interventions to reduce the
loss of efficiency resulting from decentralized control.
This leads to the topic of mechanism design. 
Such a line of work also defines a very important and interesting future
research direction.
 

\section*{Acknowledgement}
This work was partially supported by the Air Force Office of Scientific Research (AFOSR) under grant FA9550-12-1-0275.

\appendix{}

\section{Proof of Proposition~\ref{Exist_Cont}} 
\label{proof_exist_cont}

\vspace{5 pt}
\noindent Since $\chi_k \rightarrow \bar{\chi}$ we must have $\chi_k^f \to \bar{\chi}^f$ and  
  \begin{equation} \label{interm_conv}
    \lim_{k \to \infty} \norm{\chi^f_{k+1} - \chi^f_k} = 0.
  \end{equation}
  By construction of $\chi_{k,f}$, \eqref{interm_conv} implies that
  \begin{equation} 
    \label{interm_conv1} \lim_{k \to \infty} \chi_{k,f} = \bar{\chi}. 
  \end{equation}
  By Step 1 of Algorithm 2, $\chi_{k+1}^f \in
  \Xi^f(\chi_{k,f}^{-f})$. Since $\chi_{k+1}^f \rightarrow
  \bar{\chi}^f$, $\chi_k^{-f} \rightarrow \bar{\chi}^{-f}$, and
  $\Xi^f(\chi_{k,f}^{-f})$ is defined by linear inequalities parametrized
  by $\chi_{k,f}^{-f}$, it is straightforward to see by continuity
  arguments that $\bar{\chi}^f \in \Xi^f (\bar{\chi}^{-f})$. In other
  words, $\bar{\chi}$ is feasible for every agent's optimization problem
  \eqref{gnep1}. 

  We claim that for each agent $f \in \mathcal{F}$
  \begin{equation*} 
    \theta^f(\bar{\chi}^f, \bar{\chi}^{-f}) \geq \theta^f(\chi^f, \bar{\chi}^{-f}), \quad \forall \ \chi^f \in \Xi^f(\bar{\chi}^{-f}). 
  \end{equation*}
  For the purposes of establishing a contradiction, let there be an
  agent $\bar{f}$ and a vector $\bar{\xi}^{\bar{f}} \in
  \Xi^{\bar{f}}(\bar{\chi}^{-\bar{f}})$ such that
  \begin{equation*} 
    \theta^{\bar{f}}(\bar{\chi}^{\bar{f}}, \bar{\chi}^{-\bar{f}}) < \theta^{\bar{f}}(\bar{\xi}^{\bar{f}}, \bar{\chi}^{-\bar{f}}). 
  \end{equation*}
  Using the linearity of the functions that define the set valued mapping
  $\Xi^{\bar{f}}(\cdot)$ we can show its inner semicontinuity relative to its
  domain (cf.~\cite{rockafellar_wetts} Chapter 5). Because
  $\bar{\chi}^{-\bar{f}} \in \text{dom}(\Xi^{\bar{f}}(\cdot))$, we then
  have 
  \begin{equation} 
    \label{innersc} 
    \underset{\xi^{-\bar{f}} \to \bar{\chi}^{-\bar{f}}}{\text{lim inf}} \ \
    \Xi(\xi^{-\bar{f}}) \supseteq \Xi(\bar{\chi}^{-\bar{f}}), 
  \end{equation} 
  where the limit in \eqref{innersc} is given by the following: 
  \begin{equation} 
    \label{innersc2} 
    \underset{\xi^{-\bar{f}} \to \bar{\chi}^{-\bar{f}}}{\text{lim inf}} \ \ \Xi(\xi^{-\bar{f}}) = \displaystyle \left\{ u^{\bar{f}} \ \vert  \ \forall \chi_k^{-\bar{f}} \to \bar{\chi}^{-\bar{f}}, \exists u^{\bar{f}}_k \to u \text{ with } u^{\bar{f}}_k \in \Xi^{\bar{f}}(\chi_k^{-\bar{f}}) \right\}. 
  \end{equation}
  Since $\bar{\xi}^{\bar{f}} \in \Xi^{\bar{f}}(\bar{\chi}^{-\bar{f}})$,
  equations~\eqref{interm_conv1}, \eqref{innersc} and \eqref{innersc2}
  allow us to construct a sequence $\xi_k^{\bar{f}} \in
  \Xi^{\bar{f}}(\chi_{k,f}^{-\bar{f}})$ such that $\xi_k^{\bar{f}} \to
  \bar{\xi}^{\bar{f}}$ as $k \rightarrow \infty$.

  Let $d^{\bar{f}} = (\bar{\xi}^{\bar{f}} - \bar{\chi}^{\bar{f}})$. Then
  by the subdifferentiality inequality for concave functions we must have
  \begin{equation} \label{subdiff} 
    \theta'^{\bar{f}}(\bar{\chi}^{\bar{f}}, \bar{\chi}^{-\bar{f}}; d^{\bar{f}}) > 0. 
  \end{equation}
  Denote by $\Phi^f$ the regularized objective function
  for agent $f$'s subproblem. In other words,
  \begin{equation*} 
    \Phi^f(\chi^f, \chi^{-f}, z) = \theta^f(\chi^f, \chi^{-f}) - \tau \norm{\chi^f - z}^2. 
  \end{equation*}
  We then have
  \begin{equation*} 
    \Phi'^f(\chi^f, \chi^{-f}, z; d^f) = \theta'^f(\chi^f, \chi^{-f}; d^f) - 2 \tau (\chi^f - z)^T d^f. 
  \end{equation*}
  \sloppypar Note that $\chi_{k+1}^{\bar{f}}$ is obtained by solving the problem
  $\mathcal{R}(\chi_{k,\bar{f}}^{-\bar{f}}, \chi_k^{\bar{f}})$. In other words,
  $\chi_{k+1}^{\bar{f}}$ maximizes $\Phi^{\bar{f}}(\xi_{k,\bar{f}}^{\bar{f}}, \chi_{k,\bar{f}}^{-\bar{f}}, \chi_k^{\bar{f}})$ over the set $\Xi^{\bar{f}}(\chi_{k,\bar{f}}^{-\bar{f}})$. Since this is a concave maximization problem, we then apply first order optimality conditions to obtain the following.
  \begin{equation}
    \begin{aligned}
      \Phi'^{\bar{f}}(\chi^{\bar{f}}_{k+1}, \chi_{k,\bar{f}}^{-\bar{f}}, \chi^{\bar{f}}_k; (\xi^{\bar{f}}_k - \chi^{\bar{f}}_{k+1})) & = \theta'^{\bar{f}}(\chi^{\bar{f}}_{k+1}, \chi_{k,\bar{f}}^{-\bar{f}}, \chi^{\bar{f}}_k ; (\xi^{\bar{f}}_k - \chi^{\bar{f}}_{k+1})) \\
      & \qquad +\ 2 \tau (\chi^{\bar{f}}_{k+1} - \chi^{-\bar{f}}_k ) (\xi^{\bar{f}}_k - \chi^{\bar{f}}_{k+1}) \\
      & \leq 0.
    \end{aligned}
  \end{equation}
  Passing to the limit $k \to \infty, \ k \in K$ and using \eqref{interm_conv1} we obtain
  \begin{equation*} 
    0 \geq \theta'^{\bar{f}}(\bar{\chi}^f, \bar{\chi}^{-f}; (\bar{\xi}^f - \bar{\chi}^f))  
  \end{equation*}
  which contradicts \eqref{subdiff}.


\bibliographystyle{plain}
\bibliography{nrl_refs}




\end{document}